\documentclass[8pt,a4paper]{article}
\usepackage[utf8]{inputenc}
\usepackage{amsfonts}
\usepackage{amsthm}
\usepackage{amsmath}
\usepackage{amssymb}
\usepackage{enumitem,kantlipsum}

\newtheorem{theorem}{Theorem}[subsection]
\theoremstyle{definition}
\newtheorem{definition}[theorem]{Definition}

\newtheorem{lemma}[theorem]{Lemma}
\newtheorem{proposition}[theorem]{Proposition}

\newtheorem{example}[theorem]{Example}
\newtheorem{remark}[theorem]{Remark}

\begin{document}
\author{Sawsan Khaskeia\\Department of Mathematics\\Ariel University, Israel\\sawsan@ariel.ac.il\and Robert Shwartz\\Department of Mathematics\\Ariel University, Israel\\robertsh@ariel.ac.il}
\date{}
\title{Generalization of the basis theorem for the $B$-type Coxeter groups}
\maketitle
    
  \begin{abstract}
The $OGS$ for  non-abelian groups is an interesting
generalization of the basis of finite abelian groups. The definition of $OGS$ states that every element of a group has a unique presentation as a product of some powers of specific generators of the group, in a specific given order. In case of the symmetric groups $S_{n}$ there is a paper of R. Shwartz, which demonstrates a strong connection between the $OGS$ and the standard Coxeter presentation of the symmetric group, which is called the standard $OGS$ of $S_n$. In this paper we generalize the standard $OGS$ of $S_n$  to the finite classical Coxeter group $B_n$. We describe the exchange laws for the generalized standard $OGS$ of $B_n$, and we connect it to the Coxeter length and the descent set of $B_n$. 
\end{abstract}

\section{Introduction}
The fundamental theorem of finitely generated abelian groups states the following:
Let $A$ be a finitely generated abelian group, then there exists generators $a_{1}, a_{2}, \ldots a_{n}$, such that every element $a$ in $A$ has a unique presentation of a form:
$$g=a_{1}^{i_{1}}\cdot a_{2}^{i_{2}}\cdots a_{n}^{i_{n}},$$
where, $i_{1}, i_{2}, \ldots, i_{n}$ are $n$ integers such that for  $1\leq k\leq n$, $0\leq i_{k}<|g_{k}|$, where $a_{k}$ has a finite order of $|a_{k}|$ in $A$, and $i_{k}\in \mathbb{Z}$, where $a_{k}$ has infinite order in $A$.
Where, the meaning of the theorem is that every abelian group $A$ is direct sum of finitely many cyclic subgroup $A_{i}$ (where $1\leq i\leq k$), for some $k\in \mathbb{N}$.

\begin{definition}\label{ogs}
Let $G$ be a non-abelian group. The ordered sequence of $n$ elements $\langle g_{1}, g_{2}, \ldots, g_{n}\rangle$ is called an $Ordered ~~ Generating ~~ System$ of the group $G$ or by shortened notation, $OGS(G)$, if every element $g\in G$ has a unique presentation in the a form
$$g=g_{1}^{i_{1}}\cdot g_{2}^{i_{2}}\cdots g_{n}^{i_{n}},$$
where, $i_{1}, i_{2}, \ldots, i_{n}$ are $n$ integers such that for  $1\leq k\leq n$, $0\leq i_{k}<r_{k}$, where  $r_{k} | |g_{k}|$  in case the order of $g_{k}$ is finite in $G$, or   $i_{k}\in \mathbb{Z}$, in case  $g_{k}$ has infinite order in $G$.
The mentioned canonical form is called $OGS$ canonical form.  For every $q>p$, $1\leq x_{q}<r_{q}$, and $1\leq x_{p}<r_{p}$ the relation
$$g_{q}^{x_{q}}\cdot g_{p}^{x_{p}} = g_{1}^{i_{1}}\cdot g_{2}^{i_{2}}\cdots g_{n}^{i_{n}},$$
is called exchange law.
\end{definition}
In contrast to finitely generated abelian groups, the existence of an $OGS$ is generally not true for every finitely generated non-abelian group. Even in case of two-generated infinite non-abelian groups it is not difficult to find counter examples. For example, the Baumslag-Solitar groups $BS(m,n)$ \cite{BS}, where $m\neq \pm1$ or $n\neq \pm1$, or most of the cases of the one-relator free product of a finite cyclic group generated by $a$, with a finite two-generated group generated by $b, c$ with the relation $a^{2}\cdot b\cdot a\cdot c=1$ \cite{S}, do not have an $OGS$. Even the question of the existence of an $OGS$ for a general finite non-abelian group is still open. Moreover, contrary to the abelian case where the exchange law is just $g_{q}\cdot g_{p}=g_{p}\cdot g_{q}$, in most of the cases of non-abelian groups with the existence of an $OGS$, the exchange laws are very complicated.  Although there are some specific non-abelian groups where the exchange laws are very convenient and have very interesting properties. A very good example of it this the symmetric group $S_{n}$. In 2001, Adin and Roichman \cite{AR} introduced a presentation of an $OGS$ canonical form  for the symmetric group $S_n$, for the hyperoctahedral group $B_n$, and for the wreath product $\mathbb{Z}_{m}\wr S_{n}$. Adin and Roichman proved that for every element of $S_n$ presented in the standard $OGS$ canonical form, the sum of the exponents of the $OGS$ equals the major-index of the permutation. Moreover, by using an $OGS$ canonical form, Adin and Roichman generalized the theorem of MacMahon \cite{Mac} to the $B$-type Coxeter group, and to the wreath product $\mathbb{Z}_{m}\wr S_{n}$. A few years later, this $OGS$ canonical form was generalized for complex reflection groups by Shwartz, Adin and Roichman \cite{SAR}. Recently, Shwartz \cite{S1} significantly extended the results of \cite{AR}, \cite{SAR}, where the $OGS$ of $S_{n}$ is strongly connected to the Coxeter length and to the descent set of the elements. Moreover, in \cite{S1}, there are described the exchange laws for the $OGS$ canonical forms of the symmetric group $S_n$, which have very interesting and surprising properties.
In the paper we try to generalize the results of \cite{S1} to the finite classical Coxeter group $B_{n}$. Similarly to the symmetric group $S_n$, the group $B_n$ can be considered as permutation group  as well. 
Therefore, we recall the notations of permutations,  the $OGS$ of $S_{n}$ and the corresponding exchange laws, from \cite{S1}.\\

\begin{definition}\label{sn}
Let $S_n$ be the symmetric group on $n$ elements, then :
\begin{itemize}
\item The symmetric group $S_n$ is an $n-1$ generated simply-laced finite Coxeter group  of order $n!$,  which has the presentation of: $$\langle s_1, s_2, \ldots, s_{n-1} | s_i^{2}=1, ~~ (s_i\cdot s_{i+1})^{3}=1, ~~(s_i\cdot s_j)^2=1 ~~for ~~|i-j|\geq 2\rangle;$$
\item The group $S_n$ can be considered as the permutation group on $n$ elements. A permutation $\pi\in S_n$ is denoted by 
$$
\pi=[\pi(1); ~\pi(2); \ldots; ~\pi(n)]
 $$
(i.e., $\pi=
[2; ~4; ~1; ~3]$ is a permutation in $S_{4}$ which satisfies $\pi(1)=2$, $\pi(2)=4$, $\pi(3)=1$, and $\pi(4)=3$);
\item Every permutation $\pi\in S_n$ can be presented in a cyclic notation, as a product of disjoint cycles of the form $(i_1, ~i_2, ~\ldots, ~i_m)$, which means $\pi(i_{k})=i_{k+1}$, for $1\leq k\leq m-1$, and $\pi(i_{m})=i_{1}$
    (i.e., The cyclic notation of $\pi=
[3; ~4; ~1; ~5; ~2]$ in $S_5$, is $(1, ~3)(2, ~4, ~5)$);
\item The Coxeter generator $s_i$ can be considered the permutation which exchanges the element $i$ with the element $i+1$, i.e., the transposition $(i, i+1)$;
\item We consider multiplication of permutations in left to right order; i.e., for every $\pi_1, \pi_2\in S_n$, $\pi_1\cdot \pi_2 (i)=\pi_2(j)$, where, $\pi_1(i)=j$ (in contrary to the notation in \cite{AR} where Adin, Roichman, and other people have considered right to left multiplication of permutations);
\item For every permutation $\pi\in S_n$, the Coxeter length $\ell(\pi)$ is the number of inversions in $\pi$, i.e., the number of different pairs $i, j$, s. t. $i<j$ and $\pi(i)>\pi(j)$;
\item For every permutation $\pi\in S_n$, the set of the locations of the descents is defined to be $$des\left(\pi\right)=\{1\leq i\leq n-1 | \pi(i)>\pi(i+1)\},$$ and $$i\in des\left(\pi\right) ~~if ~and ~only ~if ~~\ell(s_i\cdot \pi)<\ell(\pi)$$ (i.e., $i$ is a descent of $\pi$ if and only if multiplying $\pi$ by $s_i$ in the left side shortens the Coxeter length of the element.);
\item For every permutation $\pi\in S_n$, the major-index is defined to be  $$maj\left(\pi\right)=\sum_{\pi(i)>\pi(i+1)}i$$ (i.e., major-index is the sum of the locations of the descents of $\pi$.).
\item By \cite{BB} Chapter 3.4, every element $\pi$ of $S_n$ can be presented uniquely in the following normal reduced form, which we denote by $norm(\pi)$:
$$norm(\pi)=\prod_{u=1}^{n-1}\prod_{r=0}^{y_{u}-1}s_{u-r}.$$
such that $y_u$ is a non-negative integer where, $0\leq y_u\leq u$ for every $1\leq u\leq n-1$.
By our notation of $norm(\pi)$ the Coxeter length of an element $\pi$ as follow: $$\ell(\pi)=\sum_{u=1}^{n-1}y_{u}.$$
For example: Let  $m=8$, $y_2=2$, $y_4=3$, $y_5=1$, $y_8=4$, and $y_1=y_3=y_6=y_7=0$, then
$$norm(\pi)=(s_2\cdot s_1)\cdot (s_4\cdot s_3\cdot s_2)\cdot s_5\cdot (s_8\cdot s_7\cdot s_6\cdot s_5).$$
$$\ell(\pi)=2+3+1+4=10.$$
\end{itemize}
\end{definition}

\begin{theorem}\label{canonical-sn}
Let $S_n$ be the symmetric group on $n$ elements. For every $2\leq m\leq n$, define $t_{m}$ to be the product $\prod_{j=1}^{m-1}s_{j}$. The element $t_{m}$ is the permutation $$t_m=
[m; ~1; ~2; \ldots; ~m-1; ~m+1;\ldots; ~n]$$
 which is the $m$-cycle $(m, ~m-1, ~\ldots, ~1)$ in the cyclic notation of the permutation. Then, the elements $t_{n}, t_{n-1}, \ldots, t_{2}$ generates $S_n$, and every element of $S_n$ has a unique presentation in the following $OGS$ canonical form:

$$t_{2}^{i_{2}}\cdot t_{3}^{i_{3}}\cdots t_{n}^{i_{n}},~~~ where ~~~0\leq i_{k}<k ~~~for ~~~2\leq k\leq n$$
\end{theorem}

\begin{proposition}\label{exchange}
The following holds:
\\

In order to transform  the element $t_{q}^{i_{q}}\cdot t_{p}^{i_{p}}$  ($p<q$) onto the $OGS$ canonical form\\ $t_{2}^{i_{2}}\cdot t_{3}^{i_{3}}\cdots t_{n}^{i_{n}}$, i.e., according to the  standard $OGS$, one needs to use the following exchange laws:

 \[ t_{q}^{i_{q}}\cdot t_{p}^{i_{p}}=\begin{cases}
t_{i_{q}+i_{p}}^{i_q}\cdot t_{p+i_{q}}^{i_{p}}\cdot t_{q}^{i_{q}}  & q-i_{q}\geq p \\
\\
t_{i_{q}}^{p+i_{q}-q}\cdot t_{i_{q}+i_{p}}^{q-p}\cdot t_{q}^{i_{q}+i_{p}} & i_{p}\leq q-i_{q}\leq p \\
\\
t_{p+i_{q}-q}^{i_{q}+i_{p}-q}\cdot t_{i_{q}}^{p-i_{p}}\cdot t_{q}^{i_{q}+i_{p}-p}  & q-i_{q}\leq i_{p}.
\end{cases}
\]

\end{proposition}

\begin{remark}\label{exchange-2}
The standard $OGS$ canonical form of $t_{q}^{i_{q}}\cdot t_{p}^{i_{p}}$ is a product of non-zero powers of two different canonical generators if and only if $q-i_{q}=p$ or $q-i_{q}=i_{p}$, as follow:
\begin{itemize}
    \item If $q-i_q=p$ then by considering  $q-i_q\geq p$: $$t_{i_{q}+i_{p}}^{i_q}\cdot t_{p+i_{q}}^{i_{p}}\cdot t_{q}^{i_{q}}=t_{i_{q}+i_{p}}^{i_q}\cdot t_{q}^{i_p}\cdot t_{q}^{i_q}$$ and by considering  $q-i_q\leq p$:  $$t_{i_{q}}^{p+i_{q}-q}\cdot t_{i_{q}+i_{p}}^{q-p}\cdot t_{q}^{i_{q}+i_{p}}=t_{i_q}^{0}\cdot t_{i_q+i_p}^{i_q}\cdot t_{q}^{i_{q}+i_{p}};$$
    \item If $q-i_q=i_p$ then by considering  $q-i_q\geq i_p$: $$t_{i_q}^{p+i_q-q}\cdot t_{i_{q}+i_{p}}^{q-p}\cdot t_{q}^{i_{q}+i_{p}}=t_{i_q}^{p-i_p}\cdot t_q^{q-p}\cdot t_q^q=t_{i_q}^{p-i_p}\cdot t_q^{q-p}$$ and by considering  $q-i_q\leq i_p$: $$t_{p+i_{q}-q}^{i_{q}+i_{p}-q}\cdot t_{i_{q}}^{p-i_{p}}\cdot t_{q}^{i_{q}+i_{p}-p}=t_{p+i_q-q}^0\cdot t_{i_{q}}^{p-i_{p}}\cdot t_{q}^{q-p}.$$
\end{itemize}
Hence we have
 \[ t_{q}^{i_{q}}\cdot t_{p}^{i_{p}}=\begin{cases}
  t_{i_{q}+i_{p}}^{i_q}\cdot t_{q}^{i_{q}+i_{p}} & q-i_{q} = p \\
\\
  t_{i_{q}}^{p-i_{p}}\cdot t_{q}^{q-p} & q-i_{q} = i_{p}.
\end{cases}
\]
\end{remark}

Moreover, the most significant achievement of the paper \cite{S1} is the definition of the standard $OGS$ elementary factorization. By using the standard $OGS$ elementary factorization, it is possible to give a very interesting formula for the Coxeter length and a complete classification of the descent set of any arbitrary element of $S_n$. In the paper we try to generalize the standard $OGS$ elementary factorization to  the $B$-type Coxeter groups, in order to find similar properties (Coxeter length and the descent set) for the elements of the group.
Hence, we recall the definition of the standard $OGS$ elementary element and factorization for the symmetric group $S_n$ as it is defined in \cite{S1}, and theorems concerning the Coxeter length and the descent set of elements of $S_n$ as it is mentioned and proved in \cite{S1}  .

\begin{definition}\label{elementary}
Let $\pi\in S_n$, where $\pi=\prod_{j=1}^{m}t_{k_{j}}^{i_{k_{j}}}$ is presented in the standard $OGS$ canonical form, with $i_{k_{j}}>0$ for every $1\leq j\leq m$. Then, $\pi$ is called standard $OGS$ elementary element of $S_n$, if
$$\sum_{j=1}^{m}i_{k_{j}}\leq k_{1}.$$
\end{definition}

\begin{theorem}\label{theorem-elementary}\cite{S1}
Let $\pi=\prod_{j=1}^{m}t_{k_{j}}^{i_{k_{j}}}$ be a standard $OGS$ elementary element of $S_n$, presented in the standard $OGS$ canonical form, with $i_{k_{j}}>0$ for every $1\leq j\leq m$. Then, the following are satisfied:
 \begin{itemize}
\item $$\ell(\pi)=\sum_{j=1}^{m}k_{j}\cdot i_{k_{j}}-(i_{k_{1}}+i_{k_{2}}+\cdots +i_{k_{m}})^{2}=\sum_{j=1}^{m}k_{j}\cdot i_{k_{j}}-\left(maj\left(\pi\right)\right)^{2};$$
\item Every subword of $\pi$ is a standard $OGS$ elementary element too. In particular, for every two subwords $\pi_{1}$ and $\pi_{2}$ of $\pi$, such that $\pi=\pi_{1}\cdot \pi_{2}$, it is satisfied: $$\ell(\pi)=\ell(\pi_{1}\cdot \pi_{2})<\ell(\pi_{1})+\ell(\pi_{2});$$
\item $$\ell(s_r\cdot \pi)=\begin{cases} \ell(\pi)-1 & r=\sum_{j=1}^{m}i_{k_{j}} \\
\ell(\pi)+1 & r\neq \sum_{j=1}^{m}i_{k_{j}} \end{cases}.$$
i.e., $des\left(\pi\right)$ contains just one element, which means $des\left(\pi\right)=\{maj\left(\pi\right)\}$.
\end{itemize}
\end{theorem}

\begin{definition}\label{canonical-factorization-def}
Let $\pi\in S_n$. Let $z(\pi)$ be the minimal number, such that $\pi$ can be presented as a product of standard $OGS$ elementary elements, with the following conditions:
\begin{itemize}
\item $$\pi=\prod_{v=1}^{z(\pi)}\pi^{(v)}, ~~~~ where ~~~~\pi^{(v)}=\prod_{j=1}^{m^{(v)}}t_{h^{(v)}_{j}}^{\imath_{j}^{(v)}},$$
 by the presentation in the standard $OGS$ canonical form  for every $1\leq v\leq z(\pi)$ and  $1\leq j\leq m^{(v)}$ such that:
 \begin{itemize}
\item $\imath_{j}^{(v)}>0;$ \\
\item $\sum_{j=1}^{m^{(1)}}\imath_{j}^{(1)}\leq h^{(1)}_{1}$ i.e., $maj\left(\pi^{(1)}\right)\leq h^{(1)}_{1}$; \\
\item $h^{(v-1)}_{m^{(v-1)}}\leq\sum_{j=1}^{m^{(v)}}\imath_{j}^{(v)}\leq h^{(v)}_{1}$ for $2\leq v\leq z$ \\ \\
i.e., $h^{(v-1)}_{m^{(v-1)}}\leq maj\left(\pi^{(v)}\right)\leq h^{(v)}_{1} ~~ for ~~ 2\leq v\leq z$.
\end{itemize}
\end{itemize}
Then, the mentioned presentation is called \textbf{Standard $OGS$ elementary factorization} of $\pi$. Since
the factors $\pi^{(v)}$ are standard $OGS$ elementary elements, they are called standard $OGS$ elementary factors of $\pi$.
\end{definition}

\begin{theorem}\label{theorem-factorization}\cite{S1}
Let $\pi=\prod_{j=1}^{m}t_{k_{j}}^{i_{k_{j}}}$ be an element of $S_n$ presented in the standard $OGS$ canonical form, with $i_{k_{j}}>0$ for every $1\leq j\leq m$. Consider the standard $OGS$ elementary factorization of $\pi$ with all the notations used in Definition \ref{canonical-factorization-def}.
Then, the following properties hold:
\begin{itemize}
\item The standard $OGS$ elementary factorization of $\pi$ is unique, i.e., the parameters  $z(\pi)$, $m^{(v)}$ for $1\leq v\leq z(\pi)$, $h^{(v)}_{j}$, and $\imath_{j}^{(v)}$ for $1\leq j\leq m^{(v)}$, are uniquely determined by the standard $OGS$ canonical form of $\pi$, such that:
    \begin{itemize}
    \item For every $h^{(v)}_{j}$ there exists exactly one $k_{j'}$ (where, $1\leq j'\leq m$), such that $h^{(v)}_{j}=k_{j'}$;
    \item If $h^{(v)}_{j}=k_{j'}$, for some $1\leq v\leq z(\pi)$, ~$1<j<m^{(v)}$, and $1\leq j'\leq m$, then $\imath_{j}^{(v)}=i_{k_{j'}}$;
    \item If $h^{(v_{1})}_{j_{1}}=h^{(v_{2})}_{j_{2}}$, where $1\leq v_{1}<v_{2}\leq z(\pi)$, ~$1\leq j_{1}\leq m^{(v_{1})}$, and  \\ $1\leq j_{2}\leq m^{(v_{2})}$, then necessarily $v_{1}=v_{2}-1$, ~$j_{1}=m^{(v_{1})}$, ~$j_{2}=1$, and $$h^{(v_{2}-1)}_{m^{(v_{2}-1)}}=h^{(v_{2})}_{1}=maj\left(\pi_{(v_{2})}\right)=k_{j'},$$
        for some $j'$, such that $\imath_{m^{(v_{2}-1)}}^{(v_{2}-1)}+\imath_{1}^{(v_{2})}=i_{k_{j'}}$;
    \end{itemize}

\item $$\ell(s_r\cdot \pi) = \begin{cases} \ell(\pi)-1 & r=\sum_{j=1}^{m^{(v)}}\imath_{j}^{(v)} ~~for ~~ 1\leq v\leq z(\pi) \\
\ell(\pi)+1 & otherwise \end{cases}.$$
i.e., $$des\left(\pi\right)=\bigcup_{v=1}^{z(\pi)}des\left(\pi^{(v)}\right)=\{maj\left(\pi^{(v)}\right)~|~1\leq v\leq z(\pi)\};$$
\item
\begin{align*}
\ell(\pi) &= \sum_{v=1}^{z(\pi)}\ell(\pi^{(v)}) = \sum_{v=1}^{z(\pi)}\sum_{j=1}^{m^{(v)}}h^{(v)}_{j}\cdot \imath_{j}^{(v)}-\sum_{v=1}^{z(\pi)}\left(maj\left(\pi^{(v)}\right)\right)^{2} \\ &= \sum_{x=1}^{m}k_{x}\cdot i_{k_{x}}-\sum_{v=1}^{z(\pi)}\left(maj\left(\pi^{(v)}\right)\right)^{2} \\ &= \sum_{x=1}^{m}k_{x}\cdot i_{k_{x}}-\sum_{v=1}^{z(\pi)}{{\left(c^{(v)}\right)}}^{2}, ~~ where ~~ c^{(v)}\in des\left(\pi\right).
\end{align*}
\end{itemize}
\end{theorem}

The paper is organized as follow: In section \ref{gen-OGS_B-n}, we recall some important definitions and basic properties concerning the group $B_{n}$, Then we generalize the definition of the standard $OGS$ and we show the arising exchange laws. In section \ref{parabolic-subgroup}, we focus on the parabolic subgroup of $B_{n}$ which does not contain $s_{0}$ in it's reduced Coxeter presentation. It is easy to show that the subgroup is isomorphic to the symmetric group $S_n$, where we denote the subgroup  by $\dot{S}_n$. We show the presentation of $\dot{S}_n$ by the generalized standard $OGS$, and then we show the properties of the standard $OGS$ elementary elements and the the standard $OGS$ elementary factorization of the elements (as it described in \cite{S1}) in $\dot{S}_n$. In section \ref{gen-stan-factor}, we generalize the definition of the standard $OGS$ elementary factorization for the group $B_{n}$. In section \ref{cox-length}, we introduce a formula for the length function of $B_{n}$ by using the generalized standard $OGS$ elementary factorization which is defined in the previous section. In section \ref{descent-bn}, we characterize the descent set of the elements in the Coxeter group $B_{n}$, by using the generalized standard $OGS$ presentation for the elements of $B_{n}$.

\section{Generalization of the standard OGS for the Coxeter group $B_{n}$}\label{gen-OGS_B-n}

First, we recall some important definitions and basic properties concerning the group $B_{n}$, for example see \cite{BB}

\begin{definition}\label{def-bn}

Let $B_{n}$ be the  Coxeter group with $n$ generators, with the following presentation:

$$\begin{array}{r}
\left\langle s_{0}, s_{1}, \ldots, s_{n-1}\right| s_{i}^{2}=1, \left(s_{0} \cdot s_{1}\right)^{4}=1, 
\left(s_{i} \cdot s_{i+1}\right)^{3}=1 \text { for } 1 \leq i \leq n-1, 
\left.\left(s_{i} \cdot s_{j}\right)^{2}=1 \text { for }|i-j| \geq 2\right\rangle
\end{array}$$
\end{definition}

\textbf{Basic properties of $B_n$}
\\
\begin{itemize}
\item The group $B_{n}$ can be presented as a permutation group of the set $[\pm n]$, Where :$$[\pm n]=\{i \in \mathbb{Z} | 1 \leqslant i \leqslant n \quad \text{or}  \quad -n \leqslant i \leqslant -1  \}$$ With the following property:\\
$$
 \pi(-i)=-\pi(i) \quad \text { for every } \quad i \in[\pm n] . 
$$
\item  $B_{n}$, can be considered as a signed permutation group, where 
$\pi$ is uniquely determined by  $\pi(i)$ for $1\leq i\leq n$.
\item A signed permutation $\pi\in B_n$ is denoted by 
$$
[\pi(1); ~\pi(2); \ldots; ~\pi(n)]
 $$
(e.g., $\pi=
[2; ~-4; ~1; ~3]$ is a permutation in $B_{4}$ which satisfies $\pi(1)=2$, $\pi(2)=-4$, $\pi(3)=1$,  $\pi(4)=3$, and $\pi(-1)=-2$, $\pi(-2)=4$, $\pi(-3)=-1$,  $\pi(-4)=-3$)
\item 
$[|\pi(1)|; ~|\pi(2)|; \ldots; ~|\pi(n)|] $
is a permutation of $S_n$ \\
\item Similarly to $S_n$, for $1\leq i\leq n-1$, the Coxeter generators $s_i$ can be considered the permutation which exchanges the element $i$ with the element $i+1$ and additionally exchanges $-i$ with $-(i+1)$, and  $s_0$ is the permutation which exchanges the element $1$ with the element $-1$.
\item 
$B_n\simeq\mathbb{Z}_2^{n}\rtimes S_n$.

\item $|B_n|=2^n\cdot n!$.
\end{itemize}

\begin{definition}

We define des($\pi$) to be the left descent set of $\pi$.
$$s_{i} \in des(\pi) \quad if \quad \ell(s_{i} \cdot \pi) < \ell(\pi).$$

Similarly to $S_{n}$, the following properties holds in $B_{n}$:
$$
des(\pi)=\{0 \leqslant i \leqslant n-1 \mid \pi(i) > \pi(i+1)\}, \quad (\pi(0) \text{ defined to be } (0))$$ \end{definition}

Now, we recall the definition of the normal reduced form for $B_{n}$, as it is defined in \cite{BB} Chapter 3.4.

\begin{definition}\label{Normal form of $B_{n}$}

$$norm (\pi) = \prod_{i=0}^{n-1} \prod_{j=0}^{y_{i}-1} s_{|i-j|}$$ 

where, $y_i$ is a non-negative integer such that $0\leq y_i\leq 2i+1$, and the Coxeter length of $\pi$ as follow

$$\ell(\pi)=\sum_{i=0}^{n-1}y_i.$$

\end{definition}

\begin{example}
Consider the normal form of the following $\pi\in B_5$
$$norm(\pi)=s_0\cdot (s_1\cdot s_0\cdot s_1)\cdot (s_3\cdot s_2\cdot s_1\cdot s_0)\cdot (s_4\cdot s_2\cdot s_2\cdot s_1\cdot s_0\cdot s_1\cdot s_2)\cdot (s_5\cdot s_4\cdot s_3)$$

Then, $y_0=1, \ y_1=3, \ y_2=0, \ y_3=4, \ y_4=7, \ y_5=3$, and then
$$\ell(\pi)=1+3+0+4+7+3=18.$$
\end{example}

\subsection{The generalized standard OGS of $B_{n}$}
In this subsection we generalize the definition of the standard OGS (defined in \cite{S1}) for the group $B_n$, and we show the arising exchange laws.

Now, similar to the definition of $t_k$ in $S_n$ \cite{S1}, we define  $\tau_k$ in $B_n$ for $k=1,2, \ldots, n$ as follow:
\begin{definition}\label{tau}
For $k= 1,2,\ldots,n$ let $\tau_{k}$ be :
$$\tau_{k}= \prod_{j=0}^{k-1} s_{j} $$
\end{definition}

\begin{remark}
For every $1\leqslant k \leqslant n $, $\tau_{k}$ satisfies the following properties:
$$\begin{array}{l}
\tau_{k}(1)=-k \\
\tau_{k}(j)=j-1, \text { for } 2 \leq j \leq k \\
\tau_{k}(j)=j, \text { for } k+1 \leq j \leq n
\end{array}$$
\end{remark}
\begin{remark}
$$\tau_{k} = s_{0} \cdot t_{k} .$$

\end{remark}

\begin{remark}\label{tau-power}
 For $1\leq k\leq n$, let $\tau_{k}$ be an element of $B_n$ as it is defined in Definition \ref{tau}, then the following holds:
 \begin{itemize}
 \item For $i_{k}=-k$
 $$
    \tau_{k}^{i_{k}}(j) =\left\{\begin{array}{ll}
    -j & \text { for } 1 \leq j \leq k
    \\ \\
    j & \text { for } k+1 \leq j \leq n
    \end{array}\right.
    $$
\item For $0<i_k<k$
  $$
    \tau_{k}^{i_{k}}(j) =\left\{\begin{array}{ll}
    -(j-i_k+k) & \text { for } 1 \leq j \leq i_k
    \\ \\
    j-i_k & \text { for } i_k+1 \leq j \leq k
    \\ \\
    j &  \text { for } k+1 \leq j \leq n
    \end{array}\right.
    $$
 \item For $-k<i_k<0$
  $$
    \tau_{k}^{i_{k}}(j) =\left\{\begin{array}{ll}
    j-i_{k} & \text { for } 1 \leq j \leq i_k +k
    \\ \\
    -(j-i_k-k)& \text { for } i_k+k+1 \leq j \leq k
    \\ \\
    j &  \text { for } k+1 \leq j \leq n
    \end{array}\right.
    $$ 
 \end{itemize}
\end{remark}

Now, we define the generalized standard $OGS$ for the group $B_n$ as follow

\begin{theorem}\label{ogs-bn}
For $k=1, 2, \ldots, n$ let $\tau_{k}$ be the elements of $B_n$ as defined in Definition \ref{tau} then the following holds:\\

Every element $g\in B_n$ has a unique presentation in the following form:

$$\tau_{1}^{i_1}\cdot \tau_{2}^{i_2}\cdots \tau_{n}^{i_n}$$

such that,
$-k\leq i_{k}<k$.

\end{theorem}

\begin{proof}
The proof is by induction on $n$. for $n=1$, it is easy to see that $B_1$ is generated by $\tau_1$ since $B_1$ is a cyclic group of order $2$. Now, assume by induction that theorem holds for every $k$ such that $k\leq n-1$. Denote by $\dot{B}_{n-1}$ the parabolic subgroup of $B_n$ generated by $s_0, s_1, \ldots s_{n-2}$. Easy to see that $\dot{B}_{n-1}$ is isomorphic to $B_{n-1}$ (which satisfy the theorem by the induction hypothesis). Notice, that in the permutation presentation of $B_n$, every element $x\in \dot{B}_{n-1}$ satisfies $x(n)=n$. Now, consider the right cosets of $\dot{B}_{n-1}$ in $B_n$. There are $2n$ different right cosets, where every two elements $x$ and $y$ in the same right coset of $\dot{B}_{n-1}$ in $B_n$ satisfy $x(n)=y(n)$. 
Now, notice that the powers of $\tau_{n}$ satisfy the following properties:
\begin{itemize}
    \item For $0\leq i_n<n$,  $\tau_{n}^{i_n}(n)=n-i_n$;
    \item For $-n\leq i_n<0$, $\tau_{n}^{i_n}(n)=i_n$.
\end{itemize}
Hence, for $-n\leq i_n\leq n-1$, the elements $\tau_{n}^{i_n}$ gives the $2n$ different images of $n$ in the permutation presentation. Hence, for $-n\leq i_n\leq n-1$, we have the following $2n$ different right cosets of $\dot{B}_{n-1}$ in $B_n$ :
$$\dot{B}_{n-1}\tau_{n}^{i_n}.$$
Then the result of the theorem holds for $k=n$.
\end{proof}

\begin{example}
Consider the element $\pi=
[-2; ~-1; ~-4; ~-3]$ of $B_4$. Now, we construct the $OGS$ presentation of the element, as it is described in Theorem \ref{ogs-bn}.
First, notice that $\pi(4)=-3$. Hence $\tau_4^{i_4}(i_4)=-3$, which implies that  $i_4=-3$. Thus we conclude:
$$\pi\in \dot{B}_3 ~\tau_4^{-3}.$$
Now, consider $\pi(3)=-4$. Notice, that $\tau_4^{-3}(-1)=-4$. Hence, we consider  $i_3$, such that $\tau_3^{i_3}(3)=-1$, which implies $i_3=-1$. 
Thus we conclude: 
$$\pi\in \dot{B}_2 ~\tau_3^{-1}\cdot \tau_4^{-3}.$$
Continuing by the same process, consider $\pi(2)=-1$. 
Notice, that \\ $[\tau_3^{-1}\cdot \tau_4^{-3}](1)=-1$. Hence, we consider $i_2$, such that $\tau_2^{i_2}(2)=1$, which implies $i_2=2-1=1$. 
Thus we conclude: 
$$\pi = \dot{B}_1 ~\tau_2\cdot\tau_3^{-1}\cdot \tau_4^{-3}.$$
Finally, considering $\pi(1)=-2$.
Notice, that \\ $[\tau_2\cdot \tau_3^{-1}\cdot \tau_4^{-3}](-1)=-2$. Hence, we consider $i_1$, such that $\tau_1^{i_1}(1)=-1$, which implies $i_1=-1$. 
Thus we conclude: 
$$\pi = \tau_1^{-1}\cdot\tau_2\cdot\tau_3^{-1}\cdot \tau_4^{-3}.$$
\end{example}

\begin{remark}
We call the presentation of elements of $B_n$ which has been  shown in Theorem \ref{ogs-bn}  the generalized standard $OGS$ presentation of $B_n$.
\end{remark}

Now, we show some important properties of the generalized standard $OGS$ of $B_{n}$.\\
We start with the exchange laws;
\begin{proposition}
In order to transform the element $\tau_{q}^{r_{q}} \cdot \tau_{p}^{r_{p}}$ $(p<q)$ into the generalized standard $OGS$ presentation of the form $\tau_{1}^{i_{1}}\cdot \tau_{2}^{i_{2}} \cdots \tau_{n}^{i_{n}}$ , one need to use the following
exchange laws:

\begin{itemize}
    \item The case $0<r_{p}<p$:
    $$
    \tau_{q}^{r_{q}} \cdot \tau_{p}^{r_{p}}=\left\{\begin{array}{ll}
    \tau_{r_{q}}^{-r_{q}} \cdot \tau_{r_{q}+r_{p}}^{r_{q}} \cdot \tau_{p+r_{q}}^{r_{p}} \cdot \tau_{q}^{r_{q}} & q-r_{q} \geq p
    \\ \\
    \tau_{r_{q}}^{p-q} \cdot \tau_{r_{q}+r_{p}}^{q-p} \cdot \tau_{q}^{r_{q}+r_{p}} & r_{p} \leq q-r_{q} \leq p
    \\ \\
    \tau_{p+r_{q}-q}^{r_{q}+r_{p}-q} \cdot \tau_{r_{q}}^{p-r_{p}-r_{q}} \cdot \tau_{q}^{r_{q}+r_{p}-p-q} & q-r_{q} \leq r_{p}
    \end{array}\right.
    $$
    \item The case $r_{p}=-p$:
    $$
    \tau_{q}^{r_{q}} \cdot \tau_{p}^{r_{p}}=\left\{\begin{array}{ll}
    \tau_{r_{q}}^{-r_{q}} \cdot \tau_{p+r_{q}}^{-p-r_{q}} \cdot \tau_{q}^{r_{q}} & q-r_{q} \geq p 
    \\ \\
    \tau_{p+r_{q}-q}^{-p-r_{q}+q} \cdot \tau_{r_{q}}^{-r_{q}} \cdot \tau_{q}^{r_{q}-q} & q-r_{q}<p
    \end{array}\right.
    $$
    \item The case $-p < r_{p} < 0$:
    $$
    \tau_{q}^{r_{q}} \cdot \tau_{p}^{r_{p}}=\left\{\begin{array}{ll}
    \tau_{r_{q}+r_{p}+p}^{r_{q}} \cdot \tau_{p+r_{q}}^{r_{p}-r_{q}} \cdot \tau_{q}^{r_{q}} & q-r_{q} \geq p
    \\ \\
    \tau_{p+r_{q}-q}^{-p-r_{q}+q} \cdot \tau_{r_{q}}^{p+r_{q}-q} \cdot \tau_{r_{q}+r_{p}+p}^{q-p} \cdot \tau_{q}^{p+r_{p}-q+r_{q}} & r_{p}+p \leq q-r_{q} \leq p 
    \\ \\
    \tau_{p+r_{q}-q}^{r_{p}} \cdot \tau_{r_{q}}^{-r_{p}} \cdot \tau_{q}^{r_{q}+r_{p}} & q-r_{q} \leq r_{p}+p
    \end{array}\right.
    $$
\end{itemize}

\end{proposition}
\begin{proof}
By Remark $2.1.3.$ we have $$\tau_{i} = s_{0} \cdot t_{i} .$$ Then,
$$\tau_{q}^{r_{q}} \cdot \tau_{p}^{r_{p}}= (s_{0} \cdot t_{q})^{r_{q}} \cdot (s_{0} \cdot t_{p})^{r_{p}} $$

Similar to the exchange laws for the standard $OGS$ of $S_n$ as it  is presented in Proposition \ref{exchange}, we consider several cases, depending on the values of $p, q, r_p, r_q$. 
We have the following three main cases:

\begin{enumerate}
    \item $0<r_p<p$; 
    \item $r_p=-p $; 
    \item $-p<r_p<0$.
\end{enumerate}
Let start with the first case.
\begin{enumerate}
    \item \text{The case} \quad $0 < r_{p} < p$ 
    \\
  
    Similar to three cases of exchange laws which arises from the standard $OGS$ presentation of $S_n$, (as it is presented in Proposition \ref{exchange}), we have the following three subcases of exchange laws, depending on how the value of $q-r_q$ compares to the values of $p$ and $r_p$:
\begin{itemize}
    \item $q-r_q\geq p$;
    \item $r_p\leq q-r_q\leq p$;
    \item $q-r_q\leq r_p$.
\end{itemize}
Now, we start with the first subcase:

$\bullet$ $q-r_{q}\geq p$ :\\
$$\tau_{q}^{r_{q}} \cdot \tau_{p}^{r_{p}}= (s_{0} \cdot t_{q})^{r_{q}} \cdot (s_{0} \cdot t_{p})^{r_{p}} $$

Consider the permutation presentation of   $(s_{0} \cdot t_{q})^{r_{q}} \cdot (s_{0} \cdot t_{p})^{r_{p}}$ :\\
$$1 \rightarrow  -(q-r_{q}+1) \rightarrow  -(q-r_{q}+1).$$
$$r_{q} \rightarrow -q \rightarrow -q.$$
$$r_{q}+1 \rightarrow 1  \rightarrow -(p-r_{p}+1).$$
$$r_{q}+r_{p} \rightarrow r_{p} \rightarrow -(p).$$
$$r_{q}+r_{p}+1 \rightarrow r_{p}+1 \rightarrow 1.$$
$$q \rightarrow q-r_{q} \rightarrow q-r_{q}\quad \text{or}\quad p-r_{p}.$$

By the exchange laws of $S_{n}$, as it is presented  in Proposition \ref{exchange}: $$t_{q}^{r_{q}} \cdot t_{p}^{r_{p}} = t_{r_{q}+r_{p}}^{r_{q}} \cdot t_{p+r_{q}}^{r_{p}} \cdot t_{q}^{r_{q}}$$
Now, we consider the permutation presentation of $$ (s_{0} \cdot t_{r_{q}+r_{p}})^{r_{q}} \cdot (s_{0} \cdot t_{p+r_{q}})^{r_{p}} \cdot (s_{0} \cdot t_{q})^{r_{q}}$$

$$1 \rightarrow  -(r_{p}+1) \rightarrow -1 \rightarrow  q-r_{q}+1.$$
$$r_{q} \rightarrow -(r_{p}+r_{q}) \rightarrow -r_{q} \rightarrow q-r_{q}+r_{q} = q.$$
$$r_{q}+1 \rightarrow 1 \rightarrow -1 \rightarrow -(p+r_{q}-r_{p}+1) \rightarrow -(p-r_{p}+1). $$
$$r_{q}+r_{p} \rightarrow r_{p} \rightarrow -(p+r_{q}) \rightarrow -p.$$
$$r_{q}+r_{p}+1 \rightarrow r_{q}+r_{p}+1 \rightarrow r_{q}+1 \rightarrow 1. $$
$$q \rightarrow q-r_{q} \rightarrow q-r_{q}.$$

Then the permutation presentations of  $\tau_{q}^{r_q}\cdot \tau_{p}^{r_p}$ and  $\tau_{r_{q}+r_p}^{r_p} \cdot \tau_{p+r_{q}}^{r_p} \cdot \tau_{q}^{r_{q}}$
satisfy the following properties:
\begin{itemize}
    \item For $1\leq j\leq r_q$:  $\tau_{q}^{r_q}\cdot \tau_{p}^{r_p}(j)=-\tau_{r_{q}+r_{p}}^{r_q} \cdot \tau_{p+r_{q}}^{r_{p}} \cdot \tau_{q}^{r_{q}}(j)$;
    \item For $r_q+1\leq j\leq n$: $\tau_{q}^{r_q}\cdot \tau_{p}^{r_p}(j)=\tau_{r_{q}+r_{p}}^{r_q} \cdot \tau_{p+r_{q}}^{r_{p}} \cdot \tau_{q}^{r_{q}}(j)$.
\end{itemize}
Then, we get the following result:
$$\tau_{q}^{r_{q}} \cdot \tau_{p}^{r_{p}}=  \tau_{r_{q}}^{-r_{q}} \cdot \tau_{r_{q}+r_{p}}^{r_{q}} \cdot \tau_{p+r_{q}}^{r_{p}} \cdot \tau_{q}^{r_{q}},$$
in case $q-r_q\geq p$ and  $0<r_p<p$.
\\

Now, we turn to the second subcase of the case $0<r_p<p$.
\\

$\bullet$ $r_{p} \leqslant q-r_{q} \leqslant p$
$$\tau_{q}^{r_{q}} \cdot \tau_{p}^{r_{p}}= (s_{0} \cdot t_{q})^{r_{q}} \cdot (s_{0} \cdot t_{p})^{r_{p}}, $$

By the same way, we can see that the permutation presentations of  $\tau_{q}^{r_q}\cdot \tau_{p}^{r_p}$ and  $\tau_{r_{q}}^{p-q+r_q} \cdot \tau_{r_{q}+r_{p}}^{q-p} \cdot \tau_{q}^{r_{q}+r_{p}}$
satisfy the following  properties:
\begin{itemize}
    \item For $1\leq j\leq r_q$:  $\tau_{q}^{r_q}\cdot \tau_{p}^{r_p}(j)=-\tau_{r_{q}}^{p-q+r_q} \cdot \tau_{r_{q}+r_{p}}^{q-p} \cdot \tau_{q}^{r_{q}+r_{p}}(j)$;
    \item For $r_q+1\leq j\leq n$: $\tau_{q}^{r_q}\cdot \tau_{p}^{r_p}(j)=\tau_{r_{q}}^{p-q+r_q} \cdot \tau_{r_{q}+r_{p}}^{q-p} \cdot \tau_{q}^{r_{q}+r_{p}}(j)$.
\end{itemize}
Then, we get the following result:
$$\tau_{q}^{r_{q}} \cdot \tau_{p}^{r_{p}}= \tau_{r_{q}}^{p-q} \cdot \tau_{r_{q}+r_{p}}^{q-p} \cdot \tau_{q}^{r_{q}+r_{p}},$$
in case $r_p\leqslant q-r_q\leqslant p$ and  $0<r_p<p$.
\\

Now, we turn to the third subcase of the case $0<r_p<p$.
\\

$\bullet$ $q-r_{q} \leqslant r_{p}$
$$\tau_{q}^{r_{q}}\cdot \tau_{p}^{r_{p}}= (s_{0} \cdot t_{q})^{r_{q}}\cdot (s_{0} \cdot t_{p})^{r_{p}} .$$

By the same way, we can see that the permutation presentations of  $\tau_{q}^{r_q}\cdot \tau_{p}^{r_p}$ and  $\tau_{p+r_{q}-q}^{r_q+r_p-q} \cdot \tau_{r_q}^{p-r_p} \cdot \tau_{q}^{r_{q}+r_{p}-q}$
satisfy the following properties:
\begin{itemize}
    \item For $1\leq j\leq r_q$:  $\tau_{q}^{r_q}\cdot \tau_{p}^{r_p}(j)=\tau_{p+r_{q}-q}^{r_q+r_p-q} \cdot \tau_{r_q}^{p-r_p} \cdot \tau_{q}^{r_{q}+r_{p}-q}(j)$;
    \item For $r_q+1\leq j\leq n$: $\tau_{q}^{r_q}\cdot \tau_{p}^{r_p}(j)=-\tau_{p+r_{q}-q}^{r_q+r_p-q} \cdot \tau_{r_q}^{p-r_p} \cdot \tau_{q}^{r_{q}+r_{p}-q}(j)$.
   
\end{itemize}
Then, we get the result:
$$\tau_{q}^{r_{q}} \cdot \tau_{p}^{r_{p}} = \tau_{p+r_{q}-q}^{r_{q}+r_{p}-q} \cdot \tau_{r_{q}}^{p-r_{p}-r_{q}} \cdot \tau_{q}^{r_{q}+r_{p}-p-q},$$
in case $q-r_q\leqslant r_p$ and  $0<r_p<p$.
\\

    \item  \text{The case} \quad
$r_{p}=-p$.\\

The case is divided into the following  two subcases, depending on the value of $q-r_q$ compares to the values of $p$:
\begin{itemize}
    \item $q-r_q\geq p$;
    \item $q-r_q < p$.
\end{itemize}
Now, we start with the first subcase:

$\bullet$ $q-r_{q}\geq p.$
\\

Consider the permutation of $\tau_{q}^{r_{q}} \cdot \tau_{p}^{-p}$:

$$
 \tau_{q}^{r_q}=
[-(q-r_{q}+1); ~-(q-r_{q}); \ldots; ~-q; ~1; ~2;\ldots; ~q-r_q]
$$

$$
 \tau_{p}^{-p}=
[-1; ~-2; \ldots; ~-p; ~p+1; ~p+2; \ldots; ~q].
$$
Therefore the permutation presentation of $\tau_{q}^{r_q}\cdot \tau_{p}^{-p}$:
$$1 \rightarrow -(q-r_{q}+1) \rightarrow -(q-r_{q}+1).$$
$$r_{q} \rightarrow -q \rightarrow -q.$$
$$r_{q}+1 \rightarrow 1 \rightarrow -1.$$
$$r_{q}+p \rightarrow p \rightarrow -p.$$
Now, consider the permutation presentation of $\tau_{r_q}^{-r_q}\cdot \tau_{r_{q}+p}^{-(r_q+p)}$:
$$
\tau_{r_q}^{-r_q}\cdot \tau_{r_q+p}^{-(r_q+p)}=[1; ~2; \ldots; ~r_q; ~-(r_q+1); \ldots -(r_q+p); ~r_q+p+1; \ldots; ~q].      
$$
Hence we get $$\tau_{q}^{r_{q}} \cdot \tau_{p}^{-p}=\tau_{r_q}^{-r_q}\cdot \tau_{r_{q}+p}^{-(r_q+p)}\cdot \tau_{q}^{r_{q}},$$
in case $q-r_q\geq p$ and  $r_p=-p$.
\\

 $\bullet$ Now we turn to the subcase $q-r_{q}<p.$\\

Since $q-r_q<p$ is the permutation presentation of $\tau_{q}^{r_{q}} \cdot \tau_{p}^{-p}$ as follow:

$$1 \rightarrow -(q-r_{q}+1) \rightarrow q-r_{q}+1.$$
$$p+r_{q}-q \rightarrow -p \rightarrow p.$$
$$p+r_{q}-q+1 \rightarrow -(p+1) \rightarrow -(p+1).$$
$$r_{q} \rightarrow -q \rightarrow -q.$$
$$r_{q}+1 \rightarrow 1 \rightarrow -1.$$
$$q \rightarrow q-r_{q} \rightarrow -(q-r_{q}).$$

Now, considering the permutation presentation of $\tau_{r_{q}+p-q}^{-(r_q+p-q)}\cdot \tau_{r_q}^{-r_q}\cdot \tau_{q}^{-q}$:
$$
 \tau_{r_q+p-q}^{-(r_q+p-q)}\cdot \tau_{r_q}^{-r_q}\cdot \tau_{q}^{-q}=[-1; ~-2; \ldots; ~-(r_q+p-q); ~r_q+p-q+1; \ldots ~r_q; ~-(r_q+1); \ldots; ~-q].      
$$
Hence we get $$\tau_{q}^{r_{q}} \cdot \tau_{p}^{-p}=\tau_{r_{q}+p}^{-(r_q+p)}\cdot \tau_{r_q}^{-r_q}\cdot  \tau_{q}^{r_{q}-q},$$
in case $q-r_q<p$ and  $r_p=-p$.
\\

    \item  \text{The case} \quad $-p < r_{p} < 0$\\
    
    In this case, we have the following three subcases, depending on how the value of $q-r_q$ compares to the values of $p$ and $p+r_p$:
\begin{itemize}
    \item $q-r_q\geq p$;
    \item $r_p+p\leq q-r_q\leq p$;
    \item $q-r_q\leq r_p+p$.
\end{itemize}
Now, we start with the first subcase:

$\bullet$ $q-r_{q} \geq p.$
$$\tau_{q}^{r_{q}} \cdot \tau_{p}^{r_{p}} = \tau_{q}^{r_{q}} \cdot \tau_{p}^{p+r_{p}} \cdot \tau_{p}^{-p}$$

Consider the exchange law for $\tau_{q}^{r_{q}} \cdot \tau_{p}^{p+r_{p}}$  where $0<p+r_{p}<p$ and $q-r_{q} \geq p$.

$$\tau_{q}^{r_{q}} \cdot \tau_{p}^{p+r_{p}}= \tau_{r_{q}}^{-r_{q}} \cdot \tau_{r_{q}+r_{p}+p}^{r_{q}} \cdot \tau_{p+r_{q}}^{p+r_{p}} \cdot \tau_{q}^{r_{q}}.$$
Then we get:
$$ \tau_{q}^{r_{q}} \cdot \tau_{p}^{p+r_{p}} \cdot \tau_{p}^{-p}= \tau_{r_{q}}^{-r_{q}} \cdot \tau_{r_{q}+r_{p}+p}^{r_{q}} \cdot \tau_{p+r_{q}}^{p+r_{p}} \cdot \tau_{q}^{r_{q}} \cdot \tau_{p}^{-p}$$
Then, by using the exchange law for $\tau_{q}^{r_{q}} \cdot \tau_{p}^{-p}$ where $q-r_{q} \geq p$ we get:
$$\tau_{q}^{r_{q}} \cdot \tau_{p}^{r_{p}} =  \tau_{r_{q}}^{-r_{q}} \cdot \tau_{r_{q}+r_{p}+p}^{r_{q}} \cdot \tau_{p+r_{q}}^{p+r_{p}} \cdot \tau_{r_{q}}^{-r_{q}} \cdot \tau_{p+r_{q}}^{-p-r_{q}} \cdot \tau_{q}^{r_{q}}.$$

$$ =  \tau_{r_{q}}^{-r_{q}} \cdot \tau_{r_{q}+r_{p}+p}^{r_{q}} \cdot \tau_{p+r_{p}}^{-(p+r_{p})} \cdot \tau_{p+r_{q}+r_{p}}^{-(p+r_{q}+r_{p})} \cdot \tau_{p+r_{q}}^{p+r_{p}} \cdot \tau_{p+r_{q}}^{-(p+r_{q})} \cdot \tau_{q}^{r_{q}}.$$

$$ =  \tau_{r_{q}}^{-r_{q}} \cdot \tau_{r_{q}}^{-r_{q}} \cdot \tau_{p+r_{p}+r_{q}}^{-(p+r_{p}+r_{q})} \cdot \tau_{p+r_{p}+r_{q}}^{r_{q}} \cdot \tau_{p+r_{q}+r_{p}}^{-(p+r_{q}+r_{p})} \cdot \tau_{p+r_{q}}^{p+r_{p}} \cdot \tau_{p+r_{q}}^{-(p+r_{q})} \cdot \tau_{q}^{r_{q}}.$$

Then, we get the result:
$$\tau_{q}^{r_{q}} \cdot \tau_{p}^{r_{p}} = \tau_{r_{q}+r_{p}+p}^{r_{q}} \cdot \tau_{p+r_{q}}^{r_{p}-r_{q}} \cdot \tau_{q}^{r_{q}},$$
in case $q-r_q\geq p$ and  $-p<r_p<0$.
\\

Now, we turn to the second subcase of the case $-p<r_p<0$.
\\

$\bullet$ $r_{p}+p \leqslant q-r_{q} \leqslant p.$

$$\tau_{q}^{r_{q}}= \tau_{q}^{r_{q}}\cdot \tau_{p}^{p+r_{p}} \cdot \tau_{p}^{-p}.$$

$$\tau_{q}^{r_{q}}\cdot \tau_{p}^{p+r_{p}} \cdot \tau_{p}^{-p} = \tau_{r_{q}}^{p-q} \cdot \tau_{r_{q}+p+r_{p}}^{q-p}\cdot \tau_{q}^{r_{q}+p+r_{p}}\cdot \tau_{p}^{-p}.$$

Consider the exchange law for $\tau_{q}^{r_{q}+p+r_{p}}\cdot \tau_{p}^{-p}.$  where $q-r_{q}<p$ .

$$=\tau_{r_{q}}^{p-q} \cdot \tau_{r_{q}+p+r_{p}}^{q-p}\cdot \tau_{p+r_{q}+p+r_{p}-q}^{-p-r_{q}-(p+r_{p}+q)} \cdot \tau_{r_{q}+p+r_{p}}^{-(r_{q}+p+r_{p})} \cdot \tau_{q}^{r_{q}+p+r_{p}-q}$$

Consider the exchange law for $\tau_{r_{q}+p+r_{p}}^{q-p}\cdot \tau_{p+r_{q}+p+r_{p}-q}^{-p-r_{q}-(p+r_{p}+q)}$  where $q-r_{q}=p$ .

$$=\tau_{r_{q}}^{p-q} \cdot \tau_{q-p}^{-(q-p)}\cdot \tau_{r_{q}+p+r_{p}-q}^{-(r_{q}+p+r_{p})} \cdot \tau_{r_{q}+p+r_{p}}^{q-p} \cdot \tau_{r_{q}+p+r_{p}}^{-(r_{q}+p+r_{p})} \cdot \tau_{q}^{r_{q}+p+r_{p}-q}.$$

$$\tau_{r_{q}}^{p-q} \cdot \tau_{q-p}^{-(q-p)}= \tau_{r_{q}}^{r_{q}-(q-p)}\cdot \tau_{r_{q}}^{-r_{q}} \cdot \tau_{q-p}^{-(q-p)}$$

Consider the exchange law for $\tau_{r_{q}}^{r_{q}-(q-p)}\cdot \tau_{q-p}^{-(q-p)} $, where $q-r_q=p$.

$$= \tau_{r_{q}-(q-p)}^{-(r_{q}-(q-p))} \cdot \tau_{r_{q}}^{-r_{q}} \cdot \tau_{r_{q}}^{r_{q}-(q-p)} \cdot \tau_{r_{q}}^{-r_{q}} \cdot \tau_{r_{q}+p+r_{p}}^{q-p}\cdot \tau_{q}^{r_{q}+p+r_{p}-q}.$$
Then, we get the result:
$$\tau_{q}^{r_{q}}\cdot \tau_{p}^{r_{p}} = \tau_{r_{q}-q+p}^{-r_{q}+q-p}\cdot \tau_{r_{q}}^{r_{q}+p-q}\cdot \tau_{r_{q}+p+r_{p}}^{q-p}\cdot \tau_{q}^{r_{q}+p+r_{p}-q},$$
in case $r_p+p\leqslant q-r_q\leqslant p$ and  $-p<r_p<0$.
\\

Now, we turn to the third subcase of the case $0<r_p<p$.
\\

$\bullet$ $q-r_{q} \leqslant r_{p}+p.$

$$\tau_{q}^{r_{q}}\cdot \tau_{p}^{r_{p}} = \tau_{q}^{r_{q}}\cdot \tau_{p}^{p+r_{p}}\cdot \tau_{p}^{-p}$$

By the exchange law of the case $0 \leqslant r_{p} \leqslant p$ where $q-r_{q}\leqslant r_{p}$ we get:

$$\tau_{q}^{r_{q}}\cdot \tau_{p}^{p+r_{p}}\cdot \tau_{p}^{-p}= \tau_{p+r_{q}-q}^{r_{q}+p+r_{p}-q} \cdot \tau_{r_{q}}^{p-(p+r_{p})-r_{q}} \cdot \tau_{q}^{r_{q}+p+r_{p}-p-q} \cdot \tau_{p}^{-p}.$$

$$=\tau_{p+r_{q}-q}^{r_{q}+p+r_{p}-q} \cdot \tau_{r_{q}}^{-r_{p}-r_{q}} \cdot \tau_{q}^{r_{q}+r_{p}-q}\cdot \tau_{p}^{-p}. $$

$$\tau_{q}^{r_{q}+r_{p}-q} \cdot \tau_{p}^{-p}= \tau_{q}^{r_{q}+r_{p}} \cdot \tau_{q}^{-q} \cdot \tau_{p}^{-p}.
$$

Consider the exchange law for $\tau_{q}^{r_{q}+r_{p}} \cdot \tau_{p}^{-p} $ where, $q-r_q <p$

$$=\tau_{p+r_{q}-q}^{r_{q}+p+r_{p}-q} \cdot \tau_{r_{q}}^{-r_{p}-r_{q}} \cdot \tau_{p+r_{q}+r_{p}-q}^{-(p+r_{q}+r_{p}-q)}\cdot \tau_{r_{q}+r_{p}}^{-(r_{q}+r_{p})} \cdot \tau_{q}^{r_{q}+r_{p}} \cdot \tau_{q}^{-q}\cdot \tau_{q}^{-q}$$

$$\tau_{r_{q}}^{-r_{p}-r_{q}} \cdot \tau_{p+r_{q}+r_{p}-q}^{-(p+r_{q}+r_{p}-q)}= \tau_{r_{q}}^{-r_{p}} \cdot \tau_{r_{q}}^{-r_{q}}\cdot \tau_{p+r_{q}+r_{p}-q}^{-(p+r_{q}+r_{p}-q)}.  $$

Consider the exchange law for $\tau_{r_{q}}^{-r_{p}} \cdot \tau_{p+r_{q}+r_{p}-q}^{-(p+r_{q}+r_{p}-q)} $ where $q-r_q >p$

$$=\tau_{p+r_{q}-q}^{r_{q}+p+r_{p}-q} \cdot \tau_{-r_{p}}^{r_{p}} \cdot \tau_{p+r_{q}+r_{p}-q-r_{p}}^{-(p+r_{q}-q)} \cdot \tau_{r_{q}}^{-r_{p}} \cdot \tau_{r_{q}}^{-r_{q}} \cdot \tau_{r_{q}+r_{p}}^{-(r_{q}+r_{p})} \cdot \tau_{q}^{r_{q}+r_{p}}.$$

$$=\tau_{p+r_{q}-q}^{r_{q}+p+r_{p}-q} \cdot \tau_{-r_{p}}^{r_{p}} \cdot \tau_{p+r_{q}-q}^{-(p+r_{q}-q)} \cdot \tau_{r_{q}}^{-r_{p}} \cdot \tau_{r_{q}}^{-r_{q}} \cdot \tau_{r_{q}+r_{p}}^{-(r_{q}+r_{p})} \cdot \tau_{q}^{r_{q}+r_{p}}.$$

Consider the exchange law $\tau_{r_{q}}^{-r_{p}} \cdot \tau_{r_{q}+r_{p}}^{-(r_{q}+r_{p})}$ where $q-r_q=p$

$$=\tau_{p+r_{q}-q}^{r_{q}+p+r_{p}-q} \cdot \tau_{-r_{p}}^{r_{p}} \cdot \tau_{p+r_{q}-q}^{-(p+r_{q}-q)} \cdot \tau_{-r_{p}}^{r_{p}} \cdot \tau_{r_{q}+r_{p}-r_{p}}^{-r_{q}} \cdot \tau_{r_{q}}^{-r_{p}} \cdot \tau_{r_{q}}^{-r_{q}}\cdot \tau_{q}^{r_{q}+r_{p}}.$$

Consider the exchange law for $\tau_{p+r_{q}-q}^{r_q+p+r_p-q} \cdot \tau_{-r_{p}}^{r_{p}}$ where $q-r_q=p$

$$=\tau_{r_{q}+p+r_{p}-q}^{-(r_{q}+p+r_{p}-q)} \cdot \tau_{-r_{p}+r_{q}+p+r_{p}-q}^{-(r_{q}+p-q)} \cdot \tau_{p+r_{q}-q}^{r_{q}+p+r_{p}-q} \cdot \tau_{p+r_{q}-q}^{-(p+r_{q}-q)} \cdot \tau_{-r_{p}}^{r_{p}} \cdot  \tau_{r_{q}}^{-r_{q}} \cdot \tau_{r_{q}}^{-r_{p}}\cdot \tau_{r_{q}}^{-r_{q}}\cdot  \tau_{q}^{r_{q}+r_{p}}.$$

$$=\tau_{r_{q}+p+r_{p}-q}^{-(r_{q}+p+r_{p}-q)}  \cdot \tau_{p+r_{q}-q}^{r_{q}+p+r_{p}-q} \cdot \tau_{-r_{p}}^{r_{p}} \cdot   \tau_{r_{q}}^{-r_{p}}\cdot  \tau_{q}^{r_{q}+r_{p}}.$$

Consider the exchange law for $\tau_{p+r_{q}-q}^{r_{q}+p+r_{p}-q} \cdot \tau_{-r_{p}}^{r_{p}}$ where $q-r_q=p.$

$$=\tau_{r_{q}+p+r_{p}-q}^{-(r_{q}+p+r_{p}-q)} \cdot \tau_{r_{q}+p+r_{p}-q}^{-(r_{q}+p+r_{p}-q)} \cdot \tau_{-r_{p}+r_{q}+r_{p}+p-q}^{-(r_{q}+p-q)}\cdot \tau_{p+r_{q}-q}^{r_{q}+p+r_{p}-q} \cdot \tau_{r_{q}}^{-r_{p}} \cdot \tau_{q}^{r_{q}+r_{p}}.$$

$$=\tau_{r_{q}+p-q}^{-(r_{q}+p-q)} \cdot \tau_{p+r_{q}-q}^{r_{q}+p+r_{p}-q} \cdot \tau_{r_{q}}^{-r_{p}} \cdot \tau_{q}^{r_{q}+r_{p}}.$$
Then, we get the result:
$$\tau_{q}^{r_q} \cdot \tau_{p}^{r_p}=\tau_{r_{q}+p-q}^{r_{p}} \cdot \tau_{r_{q}}^{-r_{p}} \cdot \tau_{q}^{r_{q}+r_{p}},$$
in case $q-r_q\leqslant r_p+p$ and  $-p<r_p<0$.
\end{enumerate}

\end{proof}
 
 \section{The presentation of the parabolic subgroup $\dot{S}_{n}$ of $B_{n}$ by the generalized standard OGS}\label{parabolic-subgroup}

This section deals with the properties of the parabolic subgroup of $B_{n}$ which does not contain $s_{0}$ in its reduced Coxeter presentation.\\

\begin{definition} \label{SnBn}
Denote by $\dot{S}_n$ the parabolic subgroup of $B_n$ which is generated by $s_1, s_2, \ldots, s_{n-1}$ (i.e, The elements of $B_n$ which can be written without any occurrence of $s_0$).
\end{definition}

\begin{remark}
By the Coxeter diagram of $B_n$ one can easily  see that the subgroup $\dot{S}_n$  which is defined in Definition \ref{SnBn} is isomorphic to the symmetric group $S_n$.
\end{remark}

Now, we show the generalized standard $OGS$ decomposition of the elements of the subgroup $\dot{S}_{n}$ of $B_n$ by using the following two lemmas\\

\begin{lemma}\label{tau-1+1}
 
Consider the elements of $B_{n}$ are presented by the generalized standard OGS presentation, as it is described in Theorem \ref{ogs-bn}, then the following holds:\\

$\bullet$ $\tau_{k}^{-1} \cdot \tau_{k+r}$ is the element $\prod_{j=k}^{k+r-1} s_{j}$ of $\dot{S}_{n}$ for every  $1\leqslant i \leqslant k$.\\
\end{lemma}
\begin{proof}
We start with the case of $r=1$: $$ \tau_{k} = s_{0}\cdot s_{1} \cdots s_{k-1}$$ 
Then,
$$\tau_{k}^{-1}\cdot \tau_{k+1} = s_{k-1}\cdots s_{1}\cdot s_{0} \cdot s_{0} \cdot s_{1} \cdots s_{k-1} \cdot s_{k}$$ 
$$\tau_{k}^{-1}\cdot \tau_{k+1} = s_{k}$$

Now, we turn to the case: $\tau_{k}^{-1}\cdot \tau_{k+r}$.\\

The normal form of the element  $\tau_{k}^{-1}\cdot \tau_{k+r}$ as follow:
$$\tau_{k}^{-1}\cdot \tau_{k+r} = (\tau_{k}^{-1}\cdot \tau_{k+1}) \cdot (\tau_{k+1}^{-1}\cdot \tau_{k+2}) \cdots (\tau_{k+r-1}^{-1}\cdot \tau_{k+r})$$
Then we conclude :
$$ \tau_{k}^{-1}\cdot \tau_{k+r} = s_{k} \cdot s_{k+1} \cdots s_{k+r-1}$$

\end{proof}
\begin{lemma}\label{tau-i+i}
 Consider the elements of $B_{n}$ are presented by the generalized the standard OGS presentation, as it is described in Theorem \ref{ogs-bn}, then the following holds:\\
 
$\bullet$ $\tau_{k_{1}}^{-i} \cdot \tau_{k_{2}}^{i}$ is the element $\prod_{j_{1}=k_{1}}^{k_{2}-1} \prod_{j_{2}=0}^{j_{1}+i-1} s_{j_{1}-j_{2}}$ of $\dot{S}_{n}$ for every $1 \leqslant i \leqslant k$.\\
\end{lemma}
\begin{proof}
The proof is by induction on the value of $i$.\\

The case of  $i=1$ has been proved in Lemma \ref{tau-1+1}\\

Assume by induction that the lemma holds for  $i=i_0$, i.e.,
$$\tau_{k_{1}}^{-i_0} \cdot \tau_{k_{2}}^{i_0}=\prod_{j_{1}=k_{1}}^{k_{2}-1} \prod_{j_{2}=0}^{j_{1}+i_0-1} s_{j_{1}-j_{2}}.$$

Now we prove the lemma for $i=i_0+1$. \\

First, notice: 

$$\tau_{k}^{-(i_0+1)} \cdot \tau_{k+1}^{i_0+1} = \tau_{k}^{-1} \cdot (\tau_{k}^{-i_0} \cdot \tau_{k+1}^{i_0}) \cdot \tau_{k+1}$$ 
 
Then by using the induction assumption and using the property $s_p\cdot s_q=s_q\cdot s_p$ for $|p-q|\geq 2$ we have,

$$\tau_{k}^{-1} \cdot (\tau_{k}^{-i_0} \cdot \tau_{k+1}^{i_0}) \cdot \tau_{k+1} = (s_{k-1} \cdots s_{1}\cdot s_{0}) \cdot (s_{k} \cdot s_{k-1} \cdots s_{k-(i_0-1)}) \cdot (s_{0}\cdot s_{1} \cdots s_{k-1} \cdot s_{k}).$$
$$=(s_{k-1} \cdots s_{k-(i_0-2)}\cdot s_{k-(i_0-1)} \cdot s_{k-i_0})\cdot ( s_{k-(i_0+1)}\cdots s_0)\cdot (s_{k} \cdot s_{k-1} \cdots$$ $$s_{k-(i_0-1)})\cdot (s_{0}\cdot s_{1} \cdots s_{k-(i_0+1)})\cdot(s_{k-i_0}\cdot s_{k-(i_0-1)}\cdots s_k).$$ 
$$= (s_{k-1} \cdots s_{k-(i_0-2)}\cdot s_{k-(i_0-1)} \cdot s_{k-i_0})\cdot  (s_{k}\cdot s_{k-1}\cdots s_{k-(i_0-1)}\cdot s_{k-i_0} \cdot s_{k-(i_0-1)} \cdot s_{k-(i_0-1)+1} \cdots s_{k}).$$
Since $$s_{k}\cdot s_{k-1}\cdots s_{k-(i_0-1)}\cdot s_{k-i_0} \cdot s_{k-(i_0-1)} \cdots s_{k}=s_{k-i_0}\cdot s_{k-(i_0-1)}\cdots s_k\cdot s_{k-1}\cdots s_{k-(i_0-1)}\cdot s_{k-i_0}$$ We get:
$$\tau_{k}^{-1} \cdot (\tau_{k}^{-i_0} \cdot \tau_{k+1}^{i_0}) \cdot \tau_{k+1} = s_{k-1} \cdot s_{k-2} \cdots s_{k-i_0} \cdot s_{k-i_0} \cdot s_{k-(i_0-1)} \cdots s_{k-1} \cdot s_{k} \cdot s_{k-1} \cdots s_{k-i_0}.$$
$$=s_{k} \cdot s_{k-1} \cdots s_{k-i_0}.$$

Hence, the lemma holds for $i=i_0+1$, and therefore it holds for every value of $i$.
\end{proof}

\begin{theorem}\label{main-ogs-sn}
The presentation of every element  $\pi\in \dot{S}_n$  by the generalized standard $OGS$ (as it is described in Theorem \ref{ogs-bn}) has the following form:
 \begin{enumerate}
     \item  $\tau_{k_{1}}^{i_{k_{1}}} \cdot \tau_{k_{2}}^{i_{k_{2}}} \cdots \tau_{k_{m}}^{i_{k_{m}}}$  where $-k_{j} \leqslant i_{k_{j}} \leqslant k_{j}-1$, for every $1 \leqslant j \leqslant m$, $i_{k_{1}} < 0$.
    \item $\sum_{j=1}^{m} i_{k_{j}} = 0 $.
    \item $0 \leqslant \sum_{j=r}^{m} i_{k_{j}} \leqslant k_{r-1}$ for\quad $2 \leqslant r \leqslant m$.\\
    
    Where, 3. is equivalent to \\  
    
    $-k_{r} \leqslant \sum_{j=1}^{r} i_{k_{j}} \leqslant 0$ \quad for \quad $1 \leqslant r \leqslant m-1$.
\end{enumerate}
\end{theorem}

\begin{proof}

 Consider $\pi \in \dot{S}_{n}$. Then the presentation of $\pi$ by the normal form as it is  defined in \cite{BB}, Chapter 3.4, is as follow: 
 
 $$\pi = (s_{r_{1}} \cdot s_{r_{1}-1} \cdots s_{r_{1}-v_{1}} )\cdot (s_{r_{2}} \cdot s_{r_{2}-1} \cdots s_{r_{2}-v_{2}}) \cdots (s_{r_{z}} \cdot s_{r_{z}-1} \cdots s_{r_{z}-v_{z}}),$$

where $z$ is the positive integer, $r_{1} < r_{2} < \cdots < r_{z}$, and $0 \leqslant v_{j} \leqslant r_{j}-1$, for every $1 \leqslant j\leqslant z$.
By Lemma \ref{tau-i+i} we have,

$$s_{r_{j}} \cdot s_{r_{j}-1}\cdots s_{r_{j}-v_{j}} = \tau_{r_{j}}^{-(v_{j}+1)}\cdot \tau_{r_{j}+1}^{v_{j}+1}$$

for every $1\leqslant j \leqslant z$. Therefore, 
\begin{equation}\label{tau-tau+1}
\pi = \tau_{r_{1}}^{-(v_{1}+1)} \cdot \tau_{r_{1}+1}^{v_{1}+1} \cdot \tau_{r_{2}}^{-(v_{2}+1)} \cdot \tau_{r_{2}+1}^{v_{2}+1} \cdots \tau_{r_{z}}^{-(v_{z}+1)}\cdot \tau_{r_{z}+1}^{v_{z}+1}.
 \end{equation}
Now, notice for the case where $r_{j}+1=r_{j+1}$ and  $v_{j}=v_{j+1}$ for some $1\leq j\leq z-1$. Then,

\begin{equation}\label{tau-j=tau-j+1}
\tau_{r_j+1}^{v_j+1}\cdot \tau_{r_{j+1}}^{-(v_{j+1}+1)}=\tau_{r_j+1}^{v_j+1}\cdot \tau_{r_j+1}^{-(v_j+1)}=1.
\end{equation}

Hence, we get that Equation \ref{tau-tau+1} is equivalent to 
\begin{equation}\label{tau-tau+1-reduce-j}
    \pi = \tau_{r_{1}}^{-(v_{1}+1)} \cdots \tau_{r_{j}}^{-(v_{j}+1)}\cdot \tau_{r_{j+1}+1}^{v_{j+1}+1}\cdots \tau_{r_{z}+1}^{v_{z}+1}.
\end{equation}

Hence, by substituting the identity which arises from Equation \ref{tau-j=tau-j+1} in Equation \ref{tau-tau+1}, for every $1\leq j\leq z-1$, such that $r_{j}+1=r_{j+1}$ and  $v_{j}=v_{j+1}$, Equation \ref{tau-tau+1} becomes of the following form:

\begin{equation}\label{tau-tau+1-reduced}
  \pi = \tau_{r_{1}}^{-(v_{1}+1)} \cdot \tau_{r_{q_1}+1}^{v_{1}+1} \cdot \tau_{r_{q_1+1}}^{-(v_{q_1}+1)} \cdot \tau_{r_{q_2}+1}^{v_{q_1}+1} \cdots \tau_{r_{q_{z'}+1}}^{-(v_{q_{z'}}+1)}\cdot \tau_{r_{z}+1}^{v_{q_{z'}}+1}.  
    \end{equation}

such that the following holds:
\begin{itemize}
    \item  $1\leq q_1$;
    \item $q_j+1\leq q_{j+1}$ for every $1\leq j\leq z'-1$; 
    \item $q_{z'}=z$.
\end{itemize}

Now, consider the generalized standard $OGS$ presentation of $\pi$ as it is described in Theorem \ref{ogs-bn}.

\begin{equation}\label{ogs-formula}
\pi=\tau_{k_{1}}^{i_{k_{1}}} \cdot \tau_{k_{2}}^{i_{k_{2}}} \cdots \tau_{k_{m}}^{i_{k_{m}}}
\end{equation}
such that  \item $-k_{j} \leqslant i_{k_{j}} \leqslant k_{j}-1$ , for every $1 \leqslant j \leqslant m$.

Then, by concerning the uniqueness of presentation of an element $\pi\in B_n$ by the generalized standard $OGS$, and  using the presentation of $\pi$ as it is presented in Equation \ref{tau-tau+1-reduced}, the following properties are satisfied.
\begin{itemize}
\item
    \begin{equation}\label{equiv0}
     \sum_{j=1}^m i_{k_j}= -(v_1+1)+(v_1+1)+\sum_{j=1}^{z'}(-(v_{q_j}+1)+(v_{q_j}+1))=0.
    \end{equation}
    Hence,  part 2 of the theorem holds.\\
    
    Now, we show  part 3 of the theorem, where we use all notations which were used from the beginning of the proof up to this point. We consider the generalized standard $OGS$ presentation of $\pi$  with the notations of Equation \ref{ogs-formula}. Then, consider the following equation. 
    \begin{equation}\label{part3}
    -k_{r} \leqslant \sum_{j=1}^{r} i_{k_{j}} \leqslant 0
    \quad for \quad 1 \leqslant r \leqslant m.
    \end{equation}
    We prove Equation \ref{part3} by induction on $r$.
    Then, we show that Equation \ref{part3} is equivalent to part 3 of the theorem.
    \item Since $k_{1}= r_{1}$, we have $-k_1\leq i_{k_{1}}= -(v_{1}+1)< 0$. Hence, Equation \ref{part3} holds for $r=1$. Now, we show Equation \ref{part3} for $r=2$, and then we prove it by induction for every $r\leq m$.
    \item  Now, we consider $k_2$ and $i_{k_2}$. Notice,  $k_{2}=r_{q_1}+1$ and by Equation \ref{tau-tau+1-reduced}, either  $$r_{q_1}+1 < r_{{q_1}+1}$$ or $$r_{q_1}+1 = r_{{q_1}+1}.$$
    First, assume  $$r_{q_1}+1 < r_{{q_1}+1}.$$  Then, $i_{k_2}=v_1+1$, and then 
    \begin{equation}\label{nequal}
    i_{k_{1}}+ i_{k_{2}} -(v_1+1)+(v_1+1)=0.
    \end{equation}
    Now, assume $$r_{q_1}+1 = r_{{q_1}+1}.$$  Then, $$\tau_{r_{q_1}+1}^{v_{1}+1} \cdot \tau_{r_{q_1+1}}^{-(v_{q_1}+1)}=\tau_{r_{q_1}+1}^{v_{1}-v_{q_1}}.$$
    Hence, we get $$i_{k_2}={v_{1}-v_{q_1}}.$$ 
    Therefore,
    $$i_{k_{1}}+i_{k_{2}}= -(v_{1}+1)+(v_{1}-v_{q_1})=-(v_{q_1}+1)<0,$$
    and since $$-(v_{q_1}+1)\geq -\tau_{r_{q_1+1}}=-\tau_{r_{q_1}+1}=-k_2,$$
    we get 
    \begin{equation}\label{equal}
    -k_2\leqslant i_{k_{1}} + i_{k_{2}}<0.
    \end{equation}
    Hence, by considering both Equations \ref{nequal} and \ref{equal}
    we get 
     $$-k_{2} \leqslant i_{k_{1}} + i_{k_{2}}\leqslant 0,$$ 
     for every $\pi\in \dot{S}_n.$
     \item Now, assume by induction on $p$, for every $1\leq p\leq p_0$,
     $$-k_{p} \leqslant \sum_{j=1}^{p} i_{k_{j}} \leqslant 0. $$ 
     Consider $k_{{p_0}+1}$ and $i_{k_{{p_0}+1}}$. Notice, either  $k_{{p_0}+1}=r_{q_y}+1$  or $k_{{p_0}+1}=r_{q_{y}+1}$ for some $1\leq y\leq z'$. 
     
     Assume $k_{{p_0}+1}=r_{q_y}+1$. Then by Equation \ref{tau-tau+1-reduced}, either  $$r_{q_y}+1 < r_{{q_y}+1}$$ or $$r_{q_y}+1 = r_{{q_{y}}+1},$$
     for some $1\leq y\leq z'$.
    First, assume  $$r_{q_y}+1 < r_{{q_y}+1}.$$  Then, $i_{k_{p_0+1}}=v_{q_{y-1}}+1$, and then by using Equation \ref{tau-tau+1-reduced}, 
    \begin{equation}\label{equiv-p}
       \sum_{j=1}^{p_0+1}i_{k_j}= -(v_1+1)+(v_1+1)+\sum_{j=1}^{y-1}(-(v_{q_j}+1)+(v_{q_j}+1))=0.
    \end{equation}
    Now, assume $$r_{q_y}+1 = r_{{q_{y}}+1}.$$  Then, 
    \begin{equation}\label{i-k-p-0+1}
    \tau_{r_{q_y}+1}^{v_{q_{y-1}}+1} \cdot \tau_{r_{q_y+1}}^{-(v_{q_y}+1)}=\tau_{r_{q_1}+1}^{v_{q_{y-1}}-v_{q_y}}\quad \text{which implies}\quad i_{k_{p_0+1}}=v_{q_{y-1}}-v_{q_y}
    \end{equation}
    By induction 
    \begin{equation}\label{induction-p0}
    -k_{p_0}\leq \sum_{j=1}^{p_0}i_{k_j}\leq 0,
    \end{equation}
    By the uniqueness of the presentation by the generalized standard $OGS$ (as it is described in Theorem \ref{ogs-bn}), and by using Equation \ref{tau-tau+1-reduced}, the following holds:
    \begin{equation}\label{equiv-p-1}
       \sum_{j=1}^{p_0}i_{k_j}= -(v_1+1)+(v_1+1)+\sum_{j=1}^{y-2}(-(v_{q_j}+1)+(v_{q_j}+1))-(v_{q_{y-1}}+1)=-(v_{q_{y-1}}+1).
    \end{equation}
    If $k_{{p_0}+1}=r_{q_y}+1$, then $k_{p_0}=r_{q_{y-1}+1}$. Then, by Equation \ref{tau-tau+1-reduced}, \\ $v_{q_{y-1}}+1\leq r_{q_{y-1}+1}$. Hence, $$v_{q_{y-1}}+1\leq k_{p_0}.$$ 
    
    Therefore,
    
   $$-k_{p_0+1} =-( r_{{q_{y}}+1})\leq -( v_{q_y}+1)=-(v_{q_{y-1}}+1)+ (v_{q_{y-1}}-v_{q_y})\leq -k_{p_0}+(v_{q_{y-1}}-v_{q_y}).$$ 
   
   By Equation \ref{i-k-p-0+1},  $i_{k_{p_0+1}}=v_{q_{y-1}}-v_{q_y}$, and then 
    by using Equation \ref{induction-p0}, we get
   
   $$-k_{p_0}+(v_{q_{y-1}}-v_{q_y})\leq \sum_{j=1}^{p_0}i_{k_j}+i_{k_{p_0+1}}=\sum_{j=1}^{p_0+1}i_{k_j}$$
   
   Then By using Equations \ref{equiv-p-1}, \ref{i-k-p-0+1} 
   
   $$\sum_{j=1}^{p_0+1}i_{k_j}=\sum_{j=1}^{p_0}i_{k_j}+i_{k_{p_0+1}}-(v_{q_{y-1}}+1)+v_{q_{y-1}}-v_{q_y}=-1-v_{q_y}\leq 0.$$
    
    Now, assume $k_{{p_0}+1}=r_{q_{y}+1}$. 
    Then by Equation \ref{tau-tau+1-reduced}, either  $$r_{q_y}+1 < r_{{q_y}+1}\quad \text{or}\quad r_{q_y}+1 = r_{{q_{y}}+1},$$
     for some $1\leq y\leq z'$.
    In case $r_{q_y}+1 = r_{{q_{y}}+1}$, we have already proved 
    $-k_{p_0+1}\leq \sum_{j=1}^{p_0+1}i_{k_j}\leq 0.$
    Hence, assume $$r_{q_y}+1 < r_{{q_y}+1}.$$
    Then, $i_{k_{p_0+1}}=-(v_{q_{y}}+1)$, and then by using Equation \ref{tau-tau+1-reduced}, 
    
       $$\sum_{j=1}^{p_0+1}i_{k_j}= -(v_1+1)+(v_1+1)+\sum_{j=1}^{y-1}(-(v_{q_j}+1)+(v_{q_j}+1))-(v_{q_{y}}+1)= -(v_{q_{y}}+1)=i_{k_{p_0+1}}.$$
    
    Therefore, 
    
    $$-k_{p_0+1}\leq i_{k_{p_0+1}}=\sum_{j=1}^{p_0+1}i_{k_j}=-(v_{q_{y}}+1)<0$$

    Hence, we get 
    \begin{equation}\label{geq-k-p}
    -k_{p} \leqslant \sum_{j=1}^{p} i_{k_{j}} \leqslant 0  \quad for \quad 1 \leqslant p \leqslant m-1.
    \end{equation}
    
    \item By Equation \ref{geq-k-p},   $-k_{p-1} \leqslant \sum_{j=1}^{p-1} i_{k_{j}} \leqslant 0$, and  by Equation \ref{equiv0}, $\sum_{j=1}^{m} i_{k_{j}}=0$.\\ 
    
    Therefore,
    
    $$0 \leqslant \sum_{j=p}^{m} i_{k_{j}}=\sum_{j=1}^{m} i_{k_{j}}- \sum_{j=1}^{p-1} i_{k_{j}}\leqslant k_{p-1}.$$
\end{itemize}

Hence, part 3 of the theorem holds.

\end{proof}

\begin{example}
Let $$\pi = \tau_{9}^{-8} \cdot \tau_{10} \cdot \tau_{11}^{-3} \cdot \tau_{13}^{10} = (\tau_{9}^{-8} \cdot \tau_{10}^{8}) \cdot (\tau_{10}^{-7} \cdot \tau_{11}^{7}) \cdot (\tau_{11}^{-10} \cdot \tau_{13}^{10})$$ 

The permutation presentation of $\pi$

$$
[9; ~-1; ~-2; ~-3; ~-4; ~-5; ~-6; ~-7; ~-8; ~10; ~11; ~12; ~13]
$$ $$\cdot 
[-10; ~1; ~2; ~3; ~4; ~5; ~6; ~7; ~8; ~ 9; ~11; ~12; ~13]
 $$ $$\cdot 
[4; ~5; ~6; ~7; ~8; ~9; ~10; ~11; ~-1; ~-2; ~-3; ~12; ~13]
$$ $$\cdot
[-4; ~-5; ~-6; ~-7; ~-8; ~-9; ~-10; ~-11; ~-12; ~-13; ~1; ~2; ~3]
$$
$$= 
[1; ~ 5; ~7; ~8; ~9; ~10; ~11; ~12; ~13; ~4; ~6; ~2; ~3].
$$

The following holds:
\begin{itemize}
    \item $k_1=9, \quad k_2=10, \quad k_3=11, \quad k_4=13$;
    \item $i_{k_{1}} =-8, \quad i_{k_2}=1, \quad i_{k_3}=-3, \quad i_{k_4}=10;$
    \item $i_{k_4}=10>0, \quad i_{k_3}+i_{k_4}=7\geq 0,\quad i_{k_2}+i_{k_3}+i_{k_4}=8\geq 0;$ 
    \item $i_{k_1}+i_{k_2}+i_{k_3}+i_{k_4}=0$. 
\end{itemize}

\end{example}

\begin{lemma}\label{tau-t}
Let $\pi$ be an element of $\dot{S}_n$. By Theorem \ref{main-ogs-sn}, the generalized standard $OGS$ presentation of $\pi$ as an element of $B_n$ has the following form $$\pi=\tau_{k_{1}}^{i_{k_{1}}} \cdot \tau_{k_{2}}^{i_{k_{2}}} \cdots \tau_{k_{m}}^{i_{k_{m}}}$$
such that
\begin{itemize}
    \item $ -k_{j} \leqslant i_{k_{j}} \leqslant k_{j}-1$, for every $1 \leqslant j \leqslant m$.
    \item $\sum_{j=1}^{m} i_{k_{j}}=0$.
    \item  $0\leqslant \sum_{j=r}^{m}i_{k_{j}}\leqslant  k_{r-1}$, for every $2\leq r\leq m$.
\end{itemize}
Then, by considering $\pi$ as an element of $S_n$, the standard $OGS$ presentation of $\pi$ as it is presented in Theorem \ref{canonical-sn},  has the following form: $$t_{k_{1}}^{i_{k_{1}}^{\prime}} \cdot t_{k_{2}}^{i_{k_{2}}^{\prime}} \cdots t_{k_{m}}^{i_{k_{m}}^{\prime}}$$where, 
\begin{itemize}
    \item  $i_{k_{j}}^{\prime}=i_{k_{j}} \quad \text{if} \quad 0\leqslant i_{k_{j}} <k_{j}.$
    \item $i_{k_{j}}^{\prime} = k_{j}+i_{k_{j}} \quad \text{if} \quad -k_{j}\leqslant i_{k_{j}} < 0.$
\end{itemize}
\end{lemma}

\begin{proof}
Consider the generalized standard $OGS$ presentation of the following element $\pi\in \dot{S}_n$, as it is described in Theorem \ref{main-ogs-sn}:

$$\pi=\tau_{k_{1}}^{i_{k_{1}}} \cdot \tau_{k_{2}}^{i_{k_{2}}} \cdots \tau_{k_{m}}^{i_{k_{m}}}$$
such that
\begin{itemize}
    \item $ -k_{j} \leqslant i_{k_{j}} \leqslant k_{j}-1$, for every $1 \leqslant j \leqslant m$.
    \item $\sum_{j=1}^{m} i_{k_{j}}=0$.
    \item  $0\leqslant \sum_{j=r}^{m}i_{k_{j}}\leqslant  k_{r-1}$, for every $2\leq r\leq m$.
\end{itemize}
Then
$$\pi= (\tau_{k_{1}}^{-(\sum_{j=2}^{m} i_{k_j})} \cdot \tau_{k_{2}}^{(\sum_{j=2}^{m} i_{k_j})}) \cdots  (\tau_{k_{m-2}}^{-(i_{k_{m-1}}+i_{k_{m}})}\cdot \tau_{k_{m-1}}^{(i_{k_{m-1}}+i_{k_{m}})})\cdot (\tau_{k_{m-1}}^{-i_{k_{m}}}\cdot \tau_{k_{m}}^{i_{k_{m}}})$$ 
Then , we conclude the following identity, where in left hand side of Equation \ref{tau_t} using Lemma \ref{tau-i+i} and in right hand side of Equation \ref{tau_t} using the standard $OGS$ presentation of $\pi$ as it is described in Theorem \ref{canonical-sn} by considering $\pi$ as an element of $S_n$.
\begin{equation}\label{tau_t}
\tau_{k_{1}}^{-j} \cdot \tau_{k_{2}}^{j} = t_{k_{1}}^{k_{1}-j} \cdot t_{k_{2}}^{j}
\end{equation}
Hence,
$$\pi=(t_{k_{1}}^{k_{1}-(\sum_{j=2}^{m} i_{k_j})} \cdot t_{k_{2}}^{\sum_{j=2}^{m} i_{k_j}})\cdots (t_{k_{m-2}}^{k_{m-2}-(\sum_{j=m-1}^{m} i_{k_j})}\cdot t_{k_{m-1}}^{\sum_{j=m-1}^{m} i_{k_j}})\cdot (t_{k_{m-1}}^{k_{m-1}-i_{k_{m}}}\cdot t_{k_{m}}^{i_{k_{m}}})=$$

$$=t_{k_{1}}^{k_{1}-(\sum_{j=2}^{m} i_{k_j})} \cdot t_{k_2}^{k_2+i_{k_2}}\cdots t_{k_m}^{k_m+i_{k_m}}$$

Since $t_{k_j}^{k_j}=1$ and $k_1\geq |i_{k_1}|=\sum_{j=2}^{m}i_{k_j}$ we get the following:
$$t_{k_{1}}^{k_{1}-(\sum_{j=2}^{m} i_{k_j})} \cdot t_{k_2}^{k_2+i_{k_2}}\cdots t_{k_m}^{k_m+i_{k_m}}=t_{k_1}^{i'_{k_1}}\cdots t_{k_m}^{i'_{k_m}}.$$ 
\end{proof}

Theorem \ref{standard-bn-sn} shows the characterization of the standard $OGS$ elementary elements of $S_n$ (defined in Definition \ref{elementary}) by the generalized standard $OGS$ of $\dot{S}_n$ according to Theorem \ref{main-ogs-sn}. 

\begin{theorem}\label{standard-bn-sn}
Let $\pi$ be  an element of $\dot{S}_{n}$ in $B_{n}$. Then, by considering $\pi$ as an element of $S_n$, $\pi$ is a standard $OGS$ elementary element as it is defined in Definition \ref{elementary} if and only if the generalized standard $OGS$ presentation of $\pi$ as it is described in Theorem \ref{ogs-bn}, has the following form: $$\pi=\tau_{k_{1}}^{i_{k_{1}}} \cdot \tau_{k_{2}}^{i_{k_{2}}} \cdots \tau_{k_{m}}^{i_{k_{m}}}, $$where, $$-k_1\leqslant i_{k_{1}} < 0 \quad , 0<i_{k_{j}} <k_j \quad \text{for} \quad j\geq 2 ,\quad \sum_{j=1}^{m} i_{k_{j}} =0.$$ 

\end{theorem}

\begin{proof}
Consider:
$$\pi=\tau_{k_{1}}^{i_{k_{1}}} \cdot \tau_{k_{2}}^{i_{k_{2}}}\cdots \tau_{k_{m}}^{i_{k_{m}}} $$
and assume $$-k_1\leqslant i_{k_{1}} < 0 \quad ,0< i_{k_{j}} <k_j \quad \text{for} \quad j\geq 2 ,\quad \sum_{j=1}^{m} i_{k_{j}} =0.$$ 
First, we prove that $\pi\in \dot{S}_n$ and then by considering $\pi$ as an element of $S_n$, we prove that $\pi$ is a standard $OGS$ elementary element (as it is defined in Definition \ref{elementary}). \\

Consider $\sum_{j=r}^{m} i_{k_{j}}$, for $2\leq r\leq m$.
Since $i_{k_r}>0$ for every $r\geq 2$, we have
$\sum_{j=r}^{m} i_{k_{j}}\leq\sum_{j=2}^{m} i_{k_{j}}$.
Since $\sum_{j=1}^{m} i_{k_{j}}=0$, we have $\sum_{j=2}^{m} i_{k_{j}}=-i_{k_1}$
Therefore, $-k_1\leq i_{k_1}<0$ implies that  $0<\sum_{j=2}^{m} i_{k_{j}}\leq k_1$.
Hence, for every $2\leq r\leq m$:

$$0<\sum_{j=r}^{m} i_{k_{j}}\leq \sum_{j=2}^{m} i_{k_{j}}\leq k_1<k_{r}.$$
Hence, by Theorem \ref{main-ogs-sn}, $\pi\in \dot{S}_n$.
Now consider the presentation of $\pi$ as an element of $S_n$ according to the standard $OGS$  presentation by using Lemma \ref{tau-t}. 

$$\pi=t_{k_1}^{i'_{k_1}}\cdot t_{k_2}^{i'_{k_2}}\cdots t_{k_m}^{i'_{k_m}}$$
such that  the following holds:
\begin{itemize}
    \item $i'_{k_r}=i_{k_r}$, since $i_{k_r}>0$ for $2\leq r\leq m$.
    \item $i'_{k_1}=k_1+i_{k_1}$, since $i_{k_1}<0$.
    \item $\sum_{j=1}^m i'_{k_j}=k_1+\sum_{j=1}^m i_{k_j}=k_1$, since $\sum_{j=1}^m i_{k_j}=0$.  
    \end{itemize}
    Hence, concerning the value of $i_{k_1}$, $\pi$ has one of the two following presentation as an element of $S_n$:
    \begin{itemize}
    \item In the case $-k_1<i_{k_1}<0$,we have $0<i'_{k_1}=k_1+i_{k_1}<k_1$, and then  the standard $OGS$ presentation (as it is presented in Theorem \ref{canonical-sn}) of $\pi$ as an element of $S_n$  as follow:
    $$\pi=t_{k_1}^{i'_{k_1}}\cdot t_{k_2}^{i'_{k_2}}\cdots t_{k_m}^{i'_{k_m}}, $$ such that $1\leq i'_{k_j}<k_j$ for every $1\leq j\leq m$, and $\sum_{j=1}^m i'_{k_j}=k_1$.
    Hence, by Definition \ref{elementary}, the element $t_{k_1}^{k_1+i_{k_1}}\cdot t_{k_2}^{i_{k_2}}\cdots t_{k_m}^{i_{k_m}}$ 
    is a standard $OGS$ elementary element;
    \item In the case $i_{k_1}=-k_1$, we have $i'_{k_1}=k_1+i_{k_1}=0$, and then the standard $OGS$ presentation (as it is presented in Theorem \ref{canonical-sn}) of $\pi$ as an element of $S_n$  as follow:  $$\pi=t_{k_1}^{k_1+i_{k_1}}\cdot t_{k_2}^{i_{k_2}}\cdots t_{k_m}^{i_{k_m}}=t_{k_2}^{i_{k_2}}\cdots t_{k_m}^{i_{k_m}},$$ such that $\sum_{j=2}^m i_{k_j}=k_1+i_{k_1}+\sum_{j=2}^m i_{k_j}=k_1<k_2$. Hence, by Definition \ref{elementary} the element $$\pi=t_{k_2}^{i_{k_2}}\cdots t_{k_m}^{i_{k_m}}$$  is a standard $OGS$ elementary element.
     \end{itemize}

Now, we turn to the proof second direction of the theorem.\\

Assume $\pi\in \dot{S}_n$, and by considering $\pi$ as an element of $S_n$, $\pi$ is a standard $OGS$ elementary element.
Then, by Definition \ref{elementary},

$$\pi=t_{k_1}^{i_{k_1}}\cdot t_{k_2}^{i_{k_2}}\cdots t_{k_m}^{i_{k_m}},$$
where, $1\leq i_{k_j}<k_j$ and  $\sum_{j=1}^m i_{k_j}\leq k_1$.Then, by \cite{S1} Theorem 28 

$$norm(\pi)=\prod_{u=\rho_1}^{k_1-1}\prod_{r=0}^{\rho_1-1}s_{u-r}\cdot \prod_{u=k_1}^{k_2-1}\prod_{r=0}^{\rho_2-1}s_{u-r}\cdot \prod_{u=k_2}^{k_3-1}\prod_{r=0}^{\rho_3-1}s_{u-r}\cdots \prod_{u=k_{m-1}}^{k_m-1}\prod_{r=0}^{\rho_m-1}s_{u-r},$$
where, $\rho_{j}=\sum_{x=j}^{m}i_{k_{x}}$ for $1\leq j\leq m$.

Now, consider $\pi$ presented in terms of generators $\tau_{k_j}$ (as it defined in Definition \ref{tau}), for $1\leq j\leq m$.
By Lemma \ref{tau-i+i}, 
$$\pi=(\tau_{maj(\pi)}^{-\rho_1}\cdot \tau_{k_1}^{\rho_1})\cdot (\tau_{k_1}^{-\rho_2}\cdot \tau_{k_2}^{\rho_2})\cdot (\tau_{k_2}^{-\rho_3}\cdot \tau_{k_3}^{\rho_3})\cdots (\tau_{k_{m-1}}^{-\rho_m}\cdot \tau_{k_m}^{\rho_m})$$
$$=\tau_{maj(\pi)}^{-\rho_1}\cdot \tau_{k_1}^{\rho_1-\rho_2}\cdot\tau_{k_2}^{\rho_2-\rho_3}\cdots \tau_{k_{m-1}}^{\rho_{m-1}-\rho_m}\cdot \tau_{k_m}^{\rho_m}.$$

Since $\rho_j=\sum_{x=j}^{m}i_{k_x}$, the following holds:
\begin{itemize}
    \item $\rho_1=maj(\pi)$;
    \item $\rho_j-\rho_{j+1}=i_{k_j}$ for $1\leq j\leq m-1$;
    \item $\rho_m=i_{k_m}$.
\end{itemize}

Now, notice that $\pi$ is a standard $OGS$ elementary element, therefore by Definition \ref{elementary}, $maj(\pi)\leq k_1$, i.e., either $maj(\pi)<k_1$ or $maj(\pi)=k_1$.

Hence, in case $maj(\pi)<k_1$, the following holds:

$$\pi= \tau_{maj(\pi)}^{-\rho_1}\cdot \tau_{k_1}^{\rho_1-\rho_2}\cdot\tau_{k_2}^{\rho_2-\rho_3}\cdots \tau_{k_{m-1}}^{\rho_{m-1}-\rho_m}\cdot \tau_{k_m}^{\rho_m}$$
$$=\tau_{maj(\pi)}^{-maj(\pi)}\cdot \tau_{k_1}^{i_{k_1}}\cdot\tau_{k_2}^{i_{k_2}}\cdots \tau_{k_{m-1}}^{i_{k_{m-1}}}\cdot \tau_{k_m}^{i_{k_m}}.$$ 
such that $0<i_{k_j}<k_j$ for $1\leq j\leq m$, and $-maj(\pi)<0$. Hence, the Theorem holds in case $maj(\pi)<k_1$.

Now, assume $maj(\pi)=k_1$, then
$$\rho_2=\sum_{j=2}^m i_{k_j}=\sum_{j=1}^m i_{k_j}-i_{k_1}=maj(\pi)-i_{k_1}=k_1-i_{k_1},$$
and then,
$$\tau_{maj(\pi)}^{-\rho_1}\cdot \tau_{k_1}^{\rho_1-\rho_2}=\tau_{k_1}^{-\rho_2}=\tau_{k_1}^{i_{k_1}-k_1}$$
Hence, the following holds:
$$\pi= \tau_{maj(\pi)}^{-\rho_1}\cdot \tau_{k_1}^{\rho_1-\rho_2}\cdot\tau_{k_2}^{\rho_2-\rho_3}\cdots \tau_{k_{m-1}}^{\rho_{m-1}-\rho_m}\cdot \tau_{k_m}^{\rho_m}$$
$$=\tau_{k_1}^{i_{k_1}-k_1}\cdot \tau_{k_2}^{i_{k_2}}\cdots \tau_{k_{m-1}}^{i_{k_{m-1}}}\cdot \tau_{k_m}^{i_{k_m}}.$$ 
such that $0<i_{k_j}<k_j$ for $2\leq j\leq m$, and $i_{k_1}-k_1<0$. Hence, the Theorem holds in case $maj(\pi)=k_1$ as well. 

\end{proof}

\begin{example}
Let $$\pi = \tau_{6}^{-5} \cdot \tau_{8}^{2} \cdot \tau_{9}^{2} \cdot \tau_{10}$$

Notice, $$i_{k_1}=-5<0, \quad i_{k_2}=2>0, \quad i_{k_3}=2>0,\quad i_{k_4}=1>0,$$
$$i_{k_1}+i_{k_2}+i_{k_3}+i_{k_4}=-5+2+2+1=0.$$

Hence, the generalized standard $OGS$ presentation of $\pi$ satisfies the conditions of Theorem \ref{standard-bn-sn}.
By Lemma \ref{tau-t}, the presentation of $\pi$ in the standard $OGS$ presentation,  by considering $\pi$ as an element of $S_n$ as follow:
$$\pi=t_{6}^{6-5}\cdot t_8^2\cdot t_9^2\cdot t_{10}=t_6\cdot t_8^2\cdot t_9^2\cdot t_{10}. $$
Notice, $maj(\pi)=1+2+2+1=6=k_1$. Hence, by Defintion \ref{elementary}, $\pi$ is a standard $OGS$ elementary element by considering $\pi$ as an elemnt of $S_n$.\\

Now, consider the permutation presentation of $\pi$.

$$\pi=
[6; ~-1; ~-2; ~-3; ~-4; ~-5; ~7; ~8; ~9; ~10]\cdot 
[-7; ~-8; ~1; ~2; ~3; ~4; ~5; ~6; ~9; ~10] 
$$ $$\cdot
[-8; ~-9; ~1; ~2; ~3; ~4; ~5; ~6; ~7; ~10] 
 \cdot
[-10; ~1; ~2; ~3; ~4; ~5; ~6; ~7; ~8; ~9]
$$
$$= 
[1; ~4; ~5; ~7; ~8; ~10; ~2; ~3; ~6; ~9]. 
$$

Notice, that the following holds:

$$\pi(1)<\pi(2)<\pi(3)<\pi(4)<\pi(5)<\pi(6), \quad \pi(6)>\pi(7),$$ $$\pi(7)<\pi(8)<\pi(9)<\pi(10).$$

Therefore, $des(\pi)=\{6\}$. By \cite{S1} Theorem 28, $\pi\in S_n$ is a standard $OGS$ elementary element if and only if $des(\pi)$ contains a single element. Hence, $\pi$ is a standard $OGS$ elementary element by considering $\pi$ as an element of $S_n$. 
\end{example}

Theorem \ref{factorization-} shows a connection between the generalized standard $OGS$ of an element $\pi\in \dot{S}_n$ as it is presented in Theorem \ref{main-ogs-sn}, and the standard $OGS$ factorization of $\pi$ as it is defined in Definition \ref{canonical-factorization-def} by considering $\pi$ as an element of $S_n$.

\begin{theorem}\label{factorization-}
 Let $\pi$ be an element of $\dot{S}_n$, such that the presentation of $\pi$ by the generalized standard $OGS$ , as it is presented in Theorem \ref{main-ogs-sn}, as follow:  $$\pi=\tau_{k_{1}}^{i_{k_{1}}} \cdot \tau_{k_{2}}^{i_{k_{2}}} \cdots \tau_{k_{m}}^{i_{k_{m}}}$$
 such that
\begin{itemize}
    \item $ -k_{j} \leqslant i_{k_{j}} \leqslant k_{j}-1$, for every $1 \leqslant j \leqslant m$.
    \item $\sum_{j=1}^{m} i_{k_{j}}=0$.
    \item  $0\leqslant \sum_{j=r}^{m}i_{k_{j}}\leqslant  k_{r-1}$, for every $2\leq r\leq m$.
\end{itemize}
 Then, consider the standard $OGS$ elementary factorization of $\pi$ as it is defined in Definition \ref{canonical-factorization-def}, where we consider $\pi$ as an element in $S_n$:  $$\pi= \prod_{v=1}^{z(\pi)} \pi^{(v)}$$  
 The value of $z(\pi)$ satisfies the following property:
 $$z(\pi) =| \{1 \leqslant j \leqslant n \quad |\quad  i_{k_{j}} < 0.\}|$$
 
\end{theorem}
\begin{proof}
Let $\pi\in \dot{S}_n$ is presented by the generalized standard $OGS$ presentation as follow:
$$\pi=\tau_{k_{1}}^{i_{k_{1}}} \cdot \tau_{k_{2}}^{i_{k_{2}}}\cdots \tau_{k_{m}}^{i_{k_{m}}}.$$ 
Denote by $q$ the number of $j$ such that $1\leq j\leq m$ and $i_{k_j}<0$, and  by $r_1, r_2\ldots, r_q$,  all the indices of $k$, such that $r_1<r_2<r_3,...<r_q$, and  $i_{k_{r_j}}<0$ for $1\leq j\leq q$.  
Since, $\pi\in \dot{S}_n$, by Theorem \ref{main-ogs-sn}, $i_{k_1}<0$. Hence, $r_1=1$.

Now, for every $1\leq \alpha\leq q$,   let define $i^{\prime}_{\alpha}$
by the following recursive algorithm:
\begin{itemize}
    \item Let define $i^{\prime}_{q}$ to be $$i^{\prime}_{q}=i_{k_m};$$
    \item For $\alpha=q-1$, let define $i^{\prime}_{q-1}$ to be
    $$i^{\prime}_{q-1}=i_{r_q}+\sum_{j=r_{q}+1}^{m-1} i_{k_{j}}+i^{\prime}_{q};$$
    \item Then for every $1\leq \alpha\leq q-2$,  let define $i^{\prime}_{\alpha}$ to be
    $$i^{\prime}_{\alpha}=i_{r_{\alpha+1}}+\sum_{j=r_{\alpha+1}+1}^{r_{\alpha+2}-1} i_{k_{j}}+i^{\prime}_{\alpha+1}.$$
\end{itemize}

Then, consider the following presentation of $\pi$

$$\pi = \tau_{k_{1}}^{i_{k_{1}}} \cdot \tau_{k_{2}}^{i_{k_{2}}}\cdots \tau_{k_{m}}^{i_{k_{m}}}$$ $$ = (\tau_{k_{r_1}}^{i_{k_{r_1}}} \cdot \tau_{k_{2}}^{i_{k_{2}}}\cdots  \tau_{k_{r_{2}-1}}^ {i_{k_{r_{2}-1}}}\cdot \tau_{k_{r_{2}}}^{i^{\prime}_{1}}) \cdot (\tau_{k_{r_{2}}}^{-\sum_{j=r_{2}+1}^{r_3-1} i_{k_{j}}-i^{\prime}_{2}} \cdot \tau_{k_{r_{2}+1}}^{i_{k_{r_{2}+1}}}  \cdots  \tau_{k_{r_{3}-1}}^ {i_{k_{r_{3}-1}}}\cdot \tau_{k_{r_{3}}}^{i^{\prime}_{2}})\cdots $$ $$ \cdots (\tau_{k_{r_{q}}}^{-\sum_{j=r_{q}+1}^{m-1} i_{k_{j}}-i^{\prime}_{q}} \cdot \tau_{k_{r_{q}+1}}^{i_{k_{r_{q}+1}}}  \cdots \tau_{k_{m-1}}^{i_{k_{m-1}}}\cdot \tau_{k_{m}}^{i^{\prime}_{q}}).$$

Since $\pi\in \dot{S}_n$,  by Theorem \ref{main-ogs-sn}, 
$\sum_{j=1}^m i_{k_j}=0$. Therefore,
$$i_{k_1}=i_{k_{r_1}}=-\sum_{j=2}^{r_2-1} i_{k_{j}}-i^{\prime}_{1}.$$
Hence, $\pi$ can be presented as follow:

$$\pi = (\tau_{k_{r_1}}^{-\sum_{j=2}^{r_2-1} i_{k_{j}}-i^{\prime}_{1}} \cdot \tau_{k_{2}}^{i_{k_{2}}}\cdots  \tau_{k_{r_{2}-1}}^ {i_{k_{r_{2}-1}}}\cdot\tau_{k_{r_{2}}}^{i^{\prime}_{1}}) \cdot (\tau_{k_{r_{2}}}^{-\sum_{j=r_{2}+1}^{r_3-1} i_{k_{j}}-i^{\prime}_{2}} \cdot \tau_{k_{r_{2}+1}}^{i_{k_{r_{2}+1}}}  \cdots \tau_{k_{r_{3}-1}}^{i_{k_{r_{3}-1}}}\cdot \tau_{k_{r_{3}}}^{i^{\prime}_{2}})\cdots$$ $$\cdots (\tau_{k_{r_{q}}}^{-\sum_{j=r_{q}+1}^{m-1} i_{k_{j}}-i^{\prime}_{q}} \cdot \tau_{k_{r_{q}+1}}^{i_{k_{r_{q}+1}}}  \cdots \tau_{k_{m-1}}^{i_{k_{m-1}}} \tau_{k_{m}}^{i^{\prime}_{q}}).$$

Now, notice, the following properties holds:
\begin{itemize}
    \item $i_{k_j}>0$ for every $1\leq j\leq m$ such that $j\neq r_p$ for $1\leq p\leq q$.
    \item Notice, $i^{\prime}_{q}=i_{k_m}$. Hence, by Theorem \ref{main-ogs-sn}, $i^{\prime}_{q}>0$.
    \item Since $\sum_{j=r_{q}+1}^{m-1} i_{k_{j}}+i^{\prime}_{q}= \sum_{j=r_{q}+1}^{m} i_{k_{j}}$, by Theorem \ref{main-ogs-sn}, $\sum_{j=r_{q}+1}^{m-1} i_{k_{j}}+i^{\prime}_{q}\leq k_{r_q}$.
    Hence, $-k_{r_q}\leq -\sum_{j=r_{q}+1}^{m-1} i_{k_{j}}-i^{\prime}_{q}<0 $.
    \item Notice, $i^{\prime}_{q-1}=i_{r_q}+\sum_{j=r_{q}+1}^{m-1} i_{k_{j}}+i^{\prime}_{q}=\sum_{j=r_q}^m i_{k_j}$. By Theorem \ref{main-ogs-sn}, $\sum_{j=r_q}^{m} i_{k_j}\geq 0$. Hence, $i^{\prime}_{q-1}\geq 0$.
   \item Now, by considering $\sum_{j=r_{\alpha+1}+1}^{r_{\alpha+2}-1} i_{k_{j}}+i^{\prime}_{\alpha+1}$  and  $i^{\prime}_{\alpha}$  for $1\leq \alpha\leq q-2$, the following properties holds (which can be proved by the same argument as it  have been proved for $\alpha=q-1$):
    \begin{itemize}
        \item $-k_{r_{\alpha+1}}\leq  -\sum_{j=r_{\alpha+1}+1}^{r_{\alpha+2}-1} i_{k_{j}}-i^{\prime}_{\alpha+1}<0$;
        \item $i^{\prime}_{\alpha}\geq 0$.
    \end{itemize}
\end{itemize}

Hence, by Theorem \ref{standard-bn-sn}, every component $\pi_{\alpha}$ of $\pi$ (where $1\leq \alpha\leq q$) of the following form
$$\pi_{\alpha}=\tau_{k_{r_{\alpha}}}^{-\sum_{j=r_{\alpha}+1}^{r_{\alpha+1}-1} i_{k_{j}}-i^{\prime}_{\alpha}} \cdot \tau_{k_{r_{\alpha}+1}}^{i_{k_{r_{\alpha}+1}}}  \cdots  \tau_{k_{r_{\alpha+1}-1}}^ {i_{k_{r_{\alpha+1}-1}}}\cdot \tau_{k_{r_{\alpha+1}}}^{i^{\prime}_{\alpha}}$$
is a standard $OGS$ elementary element by considering $\pi_{\alpha}$ as an element of $S_n$.\\

Notice the following observations:
\begin{itemize}
    \item Since $i^{\prime}_{\alpha}\geq 0$ for $1\leq \alpha\leq q-1$, we have that either $i^{\prime}_{\alpha}>0$ or $i^{\prime}_{\alpha}=0$ for $1\leq \alpha\leq q-1$. ( $i^{\prime}_{q}=i_{k_m}>0$ for every $\pi\in \dot{S}_n$).
    \item Since $\sum_{j=r_{\alpha}+1}^{r_{\alpha+1}-1}i_{k_{j}}+i^{\prime}_{\alpha}\leq k_{r_{\alpha}}$, we have that either \\ $k_{r_{\alpha}}-\sum_{j=r_{\alpha}+1}^{r_{\alpha+1}-1}i_{k_{j}}-i^{\prime}_{\alpha}>0$ or $k_{r_{\alpha}}-\sum_{j=r_{\alpha}+1}^{r_{\alpha+1}-1}i_{k_{j}}-i^{\prime}_{\alpha}=0$ for $1\leq \alpha\leq q$.
\end{itemize}

Then, by Lemma \ref{tau-t}, the presentation of the component $\pi_{\alpha}$ of  $\pi$ by the standard $OGS$ of $S_n$ is one of the four following form:

\begin{itemize}
    \item $$\pi_{\alpha}=t_{k_{r_{\alpha}}}^{k_{r_{\alpha}}-\sum_{j=r_{\alpha}+1}^{r_{\alpha+1}-1}i_{k_{j}}-i^{\prime}_{\alpha}} \cdot t_{k_{r_{\alpha}+1}}^{i_{k_{r_{\alpha}+1}}}  \cdots t_{k_{r_{\alpha+1}-1}}^ {i_{k_{r_{\alpha+1}-1}}}\cdot t_{k_{r_{\alpha+1}}}^{i^{\prime}_{\alpha}},$$
in case $i^{\prime}_{\alpha}>0$ and $\sum_{j=r_{\alpha}+1}^{r_{\alpha+1}-1}i_{k_{j}}+i^{\prime}_{\alpha}< k_{r_{\alpha}}$ (i.e., $k_{r_{\alpha}}-\sum_{j=r_{\alpha}+1}^{r_{\alpha+1}-1}i_{k_{j}}-i^{\prime}_{\alpha}>0$);
\item $$\pi_{\alpha}=t_{k_{r_{\alpha}}}^{k_{r_{\alpha}}-\sum_{j=r_{\alpha}+1}^{r_{\alpha+1}-1}i_{k_{j}}-i^{\prime}_{\alpha}} \cdot t_{k_{r_{\alpha}+1}}^{i_{k_{r_{\alpha}+1}}}  \cdots t_{k_{r_{\alpha+1}-1}}^ {i_{k_{r_{\alpha+1}-1}}},$$
in case $i^{\prime}_{\alpha}=0$ and $\sum_{j=r_{\alpha}+1}^{r_{\alpha+1}-1}i_{k_{j}}+i^{\prime}_{\alpha}< k_{r_{\alpha}}$ (i.e., $k_{r_{\alpha}}-\sum_{j=r_{\alpha}+1}^{r_{\alpha+1}-1}i_{k_{j}}-i^{\prime}_{\alpha}>0$);
\item $$\pi_{\alpha}=t_{k_{r_{\alpha}+1}}^{i_{k_{r_{\alpha}+1}}}  \cdots t_{k_{r_{\alpha+1}-1}}^ {i_{k_{r_{\alpha+1}-1}}}\cdot t_{k_{r_{\alpha+1}}}^{i^{\prime}_{\alpha}},$$
in case $i^{\prime}_{\alpha}>0$ and $\sum_{j=r_{\alpha}+1}^{r_{\alpha+1}-1}i_{k_{j}}+i^{\prime}_{\alpha}= k_{r_{\alpha}}$ (i.e., $k_{r_{\alpha}}-\sum_{j=r_{\alpha}+1}^{r_{\alpha+1}-1}i_{k_{j}}-i^{\prime}_{\alpha}=0$);
\item $$\pi_{\alpha}=t_{k_{r_{\alpha}+1}}^{i_{k_{r_{\alpha}+1}}}  \cdots t_{k_{r_{\alpha+1}-1}}^ {i_{k_{r_{\alpha+1}-1}}},$$
in case $i^{\prime}_{\alpha}=0$ and $\sum_{j=r_{\alpha}+1}^{r_{\alpha+1}-1}i_{k_{j}}+i^{\prime}_{\alpha}= k_{r_{\alpha}}$ (i.e., $k_{r_{\alpha}}-\sum_{j=r_{\alpha}+1}^{r_{\alpha+1}-1}i_{k_{j}}-i^{\prime}_{\alpha}=0$)
\end{itemize}

Now, by considering the standard $OGS$ presentation of $\pi_{\alpha}$ for $1\leq \alpha\leq q$ as an element in $S_n$ the following holds:

\begin{itemize}
    \item In all the possibly cases of $\pi_{\alpha}$, $$maj(\pi_{\alpha})=k_{r_{\alpha}}-\sum_{j=r_{\alpha}+1}^{r_{\alpha+1}-1}i_{k_{j}}-i^{\prime}_{\alpha}+\sum_{j=r_{\alpha}+1}^{r_{\alpha+1}-1}i_{k_{j}}+i^{\prime}_{\alpha}=k_{r_{\alpha}}.$$
    \item The smallest $p$ such that a non-zero power of $t_p$ is a subword of the standard $OGS$ presentation of $\pi_{\alpha}$ is $p=k_{r_{\alpha}}=maj(\pi_{\alpha})$ or \\ $p=k_{r_{\alpha}+1}>maj(\pi_{\alpha})$.
    \item The greatest $p$ such that a non-zero power of $t_p$ is a subword of the standard $OGS$ presentation of $\pi_{\alpha-1}$ (for $\alpha\geq 2$) is $p=k_{r_{\alpha}}=maj(\pi_{\alpha})$ or $p=k_{r_{\alpha}-1}<maj(\pi_{\alpha})$.
\end{itemize}

Hence, by Definition  \ref{canonical-factorization-def}, 
$$\pi=\pi_1\cdot \pi_2\cdots \pi_q$$ is a standard $OGS$ elementary factorization of $\pi$, such that every element $\pi_{\alpha}$ for $1\leq \alpha\leq q$ is the elementary factor $\pi^{(\alpha)}$ of $\pi$, where $z(\pi)=q$. 

\end{proof}

\begin{example}
Let $\pi=\tau_5^{-3}\cdot \tau_7^2\cdot \tau_8^{-4}\cdot \tau_9^4\cdot \tau_{11}^{-3}\cdot \tau_{12}^4$.
\\

Then by Theorem \ref {main-ogs-sn} $\pi\in \dot{S}_{12}$.
\\

Moreover,
\begin{itemize}
    \item $i_{k_1}=-3<0, \quad i_{k_3}=-4<0, \quad i_{k_5}=-3<0$;
    \item $i_{k_2}=2>0, \quad i_{k_4}=4>0, \quad i_{k_6}=4>0$.
\end{itemize}

Hence, by Theorem \ref{factorization-} $z(\pi)=3$ and the standard $OGS$ elementary factorization of $\pi$ as follow:

 $$\pi^{(1)}=\tau_5^{-3}\cdot \tau_7^2\cdot \tau_8,\quad \pi^{(2)}=\tau_8^{-5}\cdot \tau_9^4\cdot \tau_{11}, \quad \pi^{(3)}=\tau_{11}^{-4}\cdot \tau_{12}^4.$$

Indeed,

$$\pi^{(1)}\cdot \pi^{(2)}\cdot  \pi^{(3)}=(\tau_5^{-3}\cdot \tau_7^2\cdot \tau_8)\cdot (\tau_8^{-5}\cdot \tau_9^4\cdot \tau_{11})\cdot (\tau_{11}^{-4}\cdot \tau_{12}^4)$$
$$=\tau_5^{-3}\cdot \tau_7^2\cdot \tau_8^{-4}\cdot \tau_9^4\cdot \tau_{11}^{-3}\cdot \tau_{12}^4=\pi.$$

\end{example}

\section{Generalization of the standard OGS factorization for $B_{n}$}\label{gen-stan-factor}

There is defined In \cite{S1} a standard $OGS$ elementary factorization for the elements of the symmetric group $S_n$. This definition eases the calculation of the Coxeter length and the descent sets of the elements of $S_n$.
In this section, we generalize the definition of  the standard OGS elementary factorization for the group $B_n$.

\begin{theorem}\label{uvuv}
 Let $\pi \in B_{n}$, then $\pi$ has a unique presentation in the following form:
$u_{1} \cdot v_{1} \cdot u_{2} \cdot v_{2} \cdots u_{r}$ for some $r$, such that the following holds:
\begin{itemize}
    \item $u_{j} \in \dot{S}_{n}$, for every $1\leqslant j\leqslant r$.
    \item $v_{j} = \tau_{p_{j}}^{p_{j}}$, where for $j>j^{\prime}$, $p_{j} > p_{j^{\prime}}$.
    \item  For every $1 \leqslant j \leqslant r$, either $u_{j} = 1$ or $u_{j} = \tau_{p_{j_{1}}}^{i_{p_{j_{1}}}} \cdots \tau_{p_{j_{z_{j}}}}^{i_{p_{j_{z_{j}}}}} $ by the
generalized standard $OGS$ presentation (as it is described in Theorem \ref{ogs-bn}), where $p_{j_{1}} \geq p_{j-1}$, for every $2 \leqslant j \leqslant r$,and $p_{j_{z_{j}}} \leqslant p_{j}$ , for every $1 \leqslant j \leqslant r-1$.
\end{itemize}

\end{theorem}

\begin{proof}
Let $\pi=\tau_{k_{1}}^{i_{k_{1}}} \cdot \tau_{k_{2}}^{i_{k_{2}}} \cdots \tau_{k_{m}}^{i_{k_{m}}} $ be expressed by the generalized standard $OGS$ presentation.\\
-  First, We find $u_{r}$ and $v_{r-1}$, where we consider two cases which depend on the possibly values of $i_{k_{m}}$ as follow:
\begin{enumerate}
    \item $1\leq i_{k_m}<k_m$.
    \item $-k_m\leq i_{k_m}<0$.
\end{enumerate}
    
 - Consider case 1 (i.e.,  $1\leq i_{k_m}<k_m$).

\begin{itemize}
\item  By Theorem \ref{main-ogs-sn}, $\tau_{k_{m}}^{i_{k_{m}}}$ can be considered as a subword of an element in $\dot{S}_n$. Hence, we consider $\tau_{k_{m}}^{i_{k_{m}}}$ as a subword of $u_{r}$ and $p_{r_{z_{r}}}=k_{m}$.\\
  
  -  Now, We consider ${k_{m-1}}$ and $i_{k_{m-1}}$ in case 1 ($1\leq i_{k_m}<k_m$), where we divide the proof into the following two subcases:
 \begin{itemize}
     \item subcase 1: $k_{m-1} < i_{k_{m}}$.
     \item subcase 2: $k_{m-1} \geq i_{k_{m}}$.
 \end{itemize}
  
    \item subcase 1: If $k_{m-1} < i_{k_{m}}$, then put $u_r$ to be $u_{r}= \tau_{i_{k_{m}}}^{-i_{k_{m}}}\cdot \tau_{k_{m}}^{i_{k_{m}}}$, and put $v_{r-1}$ to be $v_{r-1} = \tau_{i_{k_{m}}}^{i_{k_{m}}} $.\\
Then, \\ 
$$\pi=\tau_{k_{1}}^{i_{k_{1}}} \cdot \tau_{k_{2}}^{i_{k_{2}}} \cdots \tau_{k_{m-1}}^{i_{k_{m-1}}} \cdot \tau_{i_{k_{m}}}^{i_{k_{m}}}\cdot \tau_{i_{k_{m}}}^{-i_{k_{m}}}\cdot \tau_{k_{m}}^{i_{k_{m}}} =\tau_{k_{1}}^{i_{k_{1}}} \cdot \tau_{k_{2}}^{i_{k_{2}}} \cdots \tau_{k_{m-1}}^{i_{k_{m-1}}} \cdot v_{r-1} \cdot u_{r}.$$

Then, we look at the subword $$\tilde{\pi}=\tau_{k_{1}}^{i_{k_{1}}} \cdot \tau_{k_{2}}^{i_{k_{2}}} \cdots \tau_{k_{m-1}}^{i_{k_{m-1}}}$$ of $\pi$ , and we consider $u_{r-1}$ and $v_{r-2}$ by applying the same algorithm as we have done for finding $u_r$ and $v_{r-1}$, by considering  $\tilde{\pi}$ instead of $\pi$.

\item subcase 2:  If $k_{m-1} \geq i_{k_{m}}$, then there are the following two possibilities:
   
   \begin{itemize}
       \item Possibility 1: $i_{k_{m}}+ i_{k_{m-1}}\geq k_{m-1}$.
       \item Possibility 2: $i_{k_{m}}+ i_{k_{m-1}}< k_{m-1}$.
   \end{itemize}
   
   \item Assume possibility 1. Then,  $i_{k_{m}}+ i_{k_{m-1}}\geq k_{m-1}$. Then we put $u_r$ to be  $u_{r}=\tau_{k_{m-1}}^{-i_{k_{m}}}\cdot \tau_{k_{m}}^{i_{k_{m}}}$  and we put $v_{r-1}$ to be  $v_{r-1}= \tau_{k_{m-1}}^{k_{m-1}}$. Then, 
   $$\pi = \tau_{k_{1}}^{i_{k_{1}}} \cdot \tau_{k_{2}}^{i_{k_{2}}} \cdots \tau_{k_{m-2}}^{i_{k_{m-2}}}\cdot  \tau_{k_{m-1}}^{{i_{k_{m-1}}+i_{k_m}-k_{m-1}}}\cdot \tau_{k_{m-1}}^{k_{m-1}}\cdot (\tau_{k_{m-1}}^{-i_{k_m}}\cdot \tau_{k_m}^{i_{k_m}})$$ $$= \tau_{k_{1}}^{i_{k_{1}}} \cdot \tau_{k_{2}}^{i_{k_{2}}} \cdots \tau_{k_{m-2}}^{i_{k_{m-2}}}\cdot  \tau_{k_{m-1}}^{{i_{k_{m-1}}+i_{k_m}-k_{m-1}}} \cdot v_{r-1} \cdot u_{r}.$$
Then, we look at the subword $$\tilde{\pi}=\tau_{k_{1}}^{i_{k_{1}}} \cdot \tau_{k_{2}}^{i_{k_{2}}} \cdots \tau_{k_{m-2}}^{i_{k_{m-2}}}\cdot \tau_{k_{m-1}}^{{i_{k_{m-1}}+i_{k_m}-k_{m-1}}}$$ of $\pi$ , and we consider $u_{r-1}$ and $v_{r-2}$ by applying the same algorithm as we have done for finding $u_r$ and $v_{r-1}$, by considering  $\tilde{\pi}$ instead of $\pi$.

   \item Assume possibility 2. Then, $i_{k_{m}}+ i_{k_{m-1}}< k_{m-1}$. Then, $\tau_{k_{m-1}}^{i_{k_{m-1}}}\cdot \tau_{k_{m}}^{i_{k_{m}}}$ is a subword of $u_{r}$ and $p_{r_{z_{r-1}}}= k_{m-1}$. Then, we consider $\tau_{k_{m-2}}$ and  $i_{k_{m-2}}$. Then, We analyze $\tau_{k_{m-2}}$ and $i_{k_{m-2}}$ by the same argument as we have analysed $\tau_{k_{m-1}}$ and  $i_{k_{m-1}}$.\\
First, dividing the proof to two main subcases
\begin{itemize}
    \item Subcase 1: $k_{m-2}< i_{k_{m}}+ i_{k_{m-1}}$.
    \item Subcase 2: $k_{m-2}\geq i_{k_{m}}+ i_{k_{m-1}}$.
     \end{itemize}
\item We analyze subcase 1 by the same algorithm as we have done in the subcase $k_{m-1}<i_{k_m}$. i.e., put $u_r$ to be $u_r=\tau_{i_{k_{m}}+ i_{k_{m-1}}}^{-i_{k_{m}}-i_{k_{m-1}}}\cdot \tau_{k_{m-1}}^{i_{k_{m-1}} }\cdot \tau_{k_{m}}^{i_{k_{m}} }$ and put $v_{r-1}$ to be $v_{r-1}= \tau_{i_{k_{m}}+ i_{k_{m-1}}}^{i_{k_{m}}+i_{k_{m-1}}}.$
\item  We analyze subcase 2 by the same algorithm as we have done in the subcase $k_{m-1}\geq i_{k_m}$ . i.e., considering two possibilities.
\begin{itemize}
    \item Possibility 1: $i_{k_{m}}+ i_{k_{m-1}}+i_{k_{m-2}}\geq k_{m-2}$.
    \item Possibility 2: $i_{k_{m}}+ i_{k_{m-1}}+i_{k_{m-2}}< k_{m-2}$.
\end{itemize}
\item Then, in case of  possibility 1, we put $u_r$ and $v_{r-1}$ by the same way  as we have done for the possibility  $i_{k_{m}}+ i_{k_{m-1}}\geq k_{m-1}$.
\item In case of possibility 2  ($i_{k_{m}}+i_{k_{m-1}}+i_{k_{m-2}} < k_{m-2}$), we conclude that $\tau_{k_{m-2}}^{i_{k_{m-2}}} \cdot \tau_{k_{m-1}}^{i_{k_{m-1}}}\cdot \tau_{k_{m}}^{i_{k_{m}}}$ is a subword of $u_{r}$ and $p_{r_{z_{r-2}}} = k_{m-2}$. Then, considering $\tau_{k_{m-3}}$ and $i_{k_{m-3}}$ by continuing  the same process
as we have done in the possibility $i_{k_{m}}+i_{k_{m-1}} < k_{m-1}$ where we have considered $\tau_{k_{m-2}}$ and $i_{k_{m-2}}$.
\item Then we continue the algorithm by the same method. 
\item Finally, we get the following conclusion from the described algorithm concerning $u_r$ and $v_{r-1}$ in case 1 ($i_{k_m}>0$). For every $1\leq x\leq m$,  Let define $\rho_x$ to be $\rho_x=\sum_{j=x}^{m} i_{k_{j}}$.
Then, we consider the largest $x$ such that $1\leq x\leq m$ such that one of the following two conditions holds
\begin{itemize}
    \item $\rho_{x+1}\leq  k_{x}\leq \rho_{x}$.
    \item $k_{x-1}<\rho_{x}<k_{x}$.
\end{itemize}
 Then $u_r$ and $v_{r-1}$ have the following presentation by the generalized standard $OGS$.
\begin{itemize}
    \item If ~~$\rho_{x+1}\leq  k_{x}\leq \rho_{x}$, then $$u_{r}=\tau_{k_{x}}^{-\rho_{x+1}}\cdot \tau_{k_{x+1}}^{i_{k_{x+1}}}\cdot\tau_{k_{x+2}}^{i_{k_{x+2}}}  \cdots \tau_{k_m}^{i_{k_m}}\quad v_{r-1}=\tau_{k_{x}}^{k_{x}}.$$
    $$\pi=\tau_{k_1}^{i_{k_1}}\cdot\tau_{k_2}^{i_{k_2}}\cdots \tau_{k_{x-1}}^{i_{k_{x-1}}}\cdot \tau_{k_x}^{\rho_x-k_x} \cdot v_{r-1}\cdot u_r .$$
    Hence, we find $u_{r-1}$ and $v_{r-2}$ by considering 
    $$\tilde{\pi}=\tau_{k_1}^{i_{k_1}}\cdot\tau_{k_2}^{i_{k_2}}\cdots \tau_{k_{x-1}}^{i_{k_{x-1}}}\cdot\tau_{k_x}^{\rho_x-k_x},$$ where we use the same method which is used to find $u_r$ and $v_{r-1}$.
    \item If ~~$k_{x-1}<\rho_{x}<k_{x}$, then
    $$u_{r}=\tau_{\rho_{x}}^{- \rho_{x}}\cdot \tau_{k_{x}}^{i_{k_{x}}}\cdot \tau_{k_{x+1}}^{i_{k_{x+1}}}\cdots \tau_{k_m}^{i_{k_m}}\quad v_{r-1}=\tau_{\rho_{x}}^{\rho_{x}}$$ $$\pi=\tau_{k_1}^{i_{k_1}}\cdot\tau_{k_2}^{i_{k_2}}\cdots \tau_{k_{x-1}}^{i_{k_{x-1}}} \cdot v_{r-1}\cdot u_r .$$
    Hence, we find $u_{r-1}$ and $v_{r-2}$ by considering 
    $$\tilde{\pi}=\tau_{k_1}^{i_{k_1}}\cdot\tau_{k_2}^{i_{k_2}}\cdots \tau_{k_{x-1}}^{i_{k_{x-1}}},$$
     where we use the same method which is used to find $u_r$ and $v_{r-1}$.
\end{itemize}

\end{itemize}

- Now, We turn to case 2 (i.e., $-k_m\leq i_{k_{m}}<0$ ).\\

If $i_{k_{m}}<0$ then, $u_{r}$ is defined to be $1$, and $v_{r-1}$ is defined to be $\tau_{k_{m}}^{-k_{m}}=\tau_{k_m}^{k_m}$. Then, 

$$\tau_{k_m}^{i_{k_m}}=\tau_{k_m}^{-k_m+(i_{k_m}+k_m)}=\tau_{k_m}^{i_{k_m}+k_m}\cdot \tau_{k_m}^{-k_m}$$
Since, $-k_m\leq i_{k_m}<0$, we get that $0\leq k_m+i_{k_m}<k_m$.
Then, $$\pi = \tau_{k_{1}}^{i_{k_{1}}} \cdot \tau_{k_{2}}^{i_{k_{2}}} \cdots \tau_{k_{m}}^{i_{k_{m}+k_{m}}} \cdot v_{r-1} \cdot u_{r},$$  where $0 \leqslant i_{k_{m}}+k_{m} < k_{m}$ .
Now, we consider the following two subcases:
\begin{itemize}
    \item Subcase 1: $i_{k_m}+k_m>0$.
    \item Subcase 2: If $i_{k_m}+k_m=0$.
\end{itemize}
Consider subcase 1. 
Then, $i_{k_m}+k_m>0$, and then we find $u_{r-1}$ by the same algorithm as we have done for finding $u_r$ in case 1 (The case of $1\leq i_{k_m}<k_m$).\\

Now, consider subcase 2. Then, $i_{k_m}+k_m=0$. Hence,  
$$\pi = \tau_{k_{1}}^{i_{k_{1}}} \cdot \tau_{k_{2}}^{i_{k_{2}}} \cdots \tau_{k_{m-1}}^{i_{k_{m-1}}}\cdot  \tau_{k_{m}}^{i_{k_{m}+k_{m}}} \cdot v_{r-1} \cdot u_{r}=\tau_{k_{1}}^{i_{k_{1}}} \cdot \tau_{k_{2}}^{i_{k_{2}}} \cdots \tau_{k_{m-1}}^{i_{k_{m-1}}}\cdot v_{r-1} \cdot u_{r}.$$
Then, we look at the subword $$\tilde{\pi}=\tau_{k_{1}}^{i_{k_{1}}} \cdot \tau_{k_{2}}^{i_{k_{2}}} \cdots \tau_{k_{m-1}}^{i_{k_{m-1}}}$$ of $\pi$ , and we consider $u_{r-1}$ and $v_{r-2}$ by applying the same algorithm as we have done for finding $u_r$ and $v_{r-1}$, by considering  $\tilde{\pi}$ instead of $\pi$.

\end{proof}

\begin{example}
Let $\pi=\tau_3^2\cdot \tau_4^3\cdot \tau_5^{-2}\cdot \tau_7^4\cdot \tau_8^2\cdot \tau_9^4$.
\\

Then, applying the algorithm as it is described in Theorem \ref{uvuv}.

First notice, $$i_{k_6}=4>0.$$ Hence $$\tau_{k_6}^{i_{k_6}}=\tau_9^4$$ is a subword of $u_r$.
Now notice, $$i_{k_6}+i_{k_5}=6,\quad  i_{k_6}+i_{k_5}+i_{k_4}=11, \quad  k_4=7,$$ 
Hence, we have:
$$i_{k_6}+i_{k_5}<k_{4}<i_{k_6}+i_{k_5}+i_{k_4}$$
Therefore, we put $u_r$ and $v_{r-1}$ as follow: $$u_r=\tau_{k_4}^{-i_{k_5}-i_{k_6}}\cdot \tau_{k_5}^{i_{k_5}}\cdot \tau_{k_6}^{i_{k_6}}=\tau_7^{-6}\cdot \tau_8^2\cdot \tau_9^4,\quad v_{r-1}=\tau_{k_4}^{k_4}=\tau_7^7.$$
Hence, we get
$$\pi=\tau_{k_1}^{i_{k_1}}\cdot \tau_{k_2}^{i_{k_2}}\cdot \tau_{k_3}^{i_{k_3}}\cdot \tau_{k_4}^{i_{k_4}+k_4-k_5-k_6}\cdot \tau_{k_4}^{k_4}\cdot (\tau_{k_4}^{-i_{k_5}-i_{k_6}}\cdot \tau_{k_5}^{i_{k_5}}\cdot \tau_{k_6}^{i_{k_6}})=$$
$$\tau_3^2\cdot \tau_4^3\cdot \tau_5^{-2}\cdot \tau_7^{3}\cdot \tau_7^{7}\cdot (\tau_7^{-6}\cdot \tau_8^2\cdot \tau_9^4).$$

Now, look at the following subword of $\pi$: $$\tau_{k_1}^{i_{k_1}}\cdot \tau_{k_2}^{i_{k_2}}\cdot\tau_{k_3}^{i_{k_3}}\cdot\tau_{k_4}^{i'_{k_4}}=\tau_3^2\cdot \tau_4^3\cdot \tau_5^{-2}\cdot \tau_7^3.$$
Where, $$i'_{k_4}=i_{k_4}+k_4-i_{k_5}-i_{k_6}.$$

Since, $$i'_{k_4}=3>0,$$

$$\tau_{k_4}^{i'_{k_4}}=\tau_7^3$$
is a subword of $u_{r-1}$.

Now, the following is satisfied:
$$i'_{k_4}+i_{k_3}=3-2=1, \quad k_2=4, \quad  i'_{k_4}+i_{k_3}+i_{k_2}=4$$
Hence,
$$i'_{k_4}+i_{k_3}<k_2\leq i'_{k_4}+i_{k_3}+i_{k_2}$$
Therefore, we put $u_{r-1}$ and $v_{r-1}$ to be as follow: $$u_{r-1}=\tau_{k_2}^{-i_{k_3}-i'_{k_4}}\cdot \tau_{k_3}^{i_{k_3}}\cdot \tau_{k_4}^{i'_{k_4}}=\tau_4^{-1}\cdot \tau_5^{-2}\cdot \tau_7^3, \quad v_{r-2}=\tau_{k_2}^{k_2}=\tau_4^4.$$

Then we get:

$$\pi=\tau_{k_1}^{i_{k_1}}\cdot \tau_{k_2}^{k_2}\cdot (\tau_{k_2}^{-i_{k_3}-i'_{k_4}}\cdot \tau_{k_3}^{i_{k_3}}\cdot \tau_{k_4}^{i'_{k_4}})\cdot \tau_{k_4}^{k_4}\cdot  (\tau_{k_4}^{-i_{k_5}-i_{k_6}}\cdot \tau_{k_5}^{i_{k_5}}\cdot \tau_{k_6}^{i_{k_6}})=$$
$$\tau_3^2\cdot \tau_4^4\cdot (\tau_4^{-1}\cdot \tau_5^{-2}\cdot \tau_7^3)\cdot \tau_7^{7}\cdot(\tau_7^{-6}\cdot \tau_8^2\cdot \tau_9^4).  $$

Now, consider the following subword of $\pi$: $$\tau_{k_1}^{i_{k_1}}=\tau_3^2.$$

Since $$k_1=2>0,$$
$$\tau_{k_1}^{i_{k_1}}=\tau_3^2$$
is a subword of $u_{r-2}$.
Now, notice $$i_{k_1}=2, \quad k_1=3.$$
Hence, $$i_{k_1}<k_1.$$
Therefore, we  get $r-3=1$, and we put $u_{r-2}=u_2$, $v_{r-3}=v_1$ and $u_{r-3}=u_1$ as follow:
$$u_{r-2}=u_2=\tau_{i_{k_1}}^{-i_{k_1}}\cdot \tau_{k_1}^{i_{k_1}}=\tau_2^{-2}\cdot \tau_3^3, \quad v_1=\tau_{i_{k_1}}^{i_{k_1}}=\tau_2^2, \quad u_1=1.$$

Hence,

$$\pi=u_1\cdot v_1\cdot u_2\cdot v_2\cdot u_3\cdot v_3\cdot u_4=\tau_2^2\cdot(\tau_2^{-2}\cdot \tau_3^2)\cdot \tau_4^4\cdot(\tau_4^{-1}\cdot \tau_5^{-2}\cdot \tau_7^3)\cdot \tau_7^{7}\cdot(\tau_7^{-6}\cdot \tau_8^2\cdot \tau_9^4).$$

Where,
$$u_1=1\quad\quad  v_1=\tau_2^2=\tau_2^{-2}\quad\quad  u_2=\tau_2^{-2}\cdot \tau_3^2\quad\quad v_2=\tau_4^4=\tau_4^{-4}$$ $$u_3=\tau_4^{-1}\cdot \tau_5^{-2}\cdot \tau_7^3\quad\quad
    v_3=\tau_7^7=\tau_7^{-7}\quad\quad  u_4=\tau_7^{-6}\cdot \tau_8^2\cdot \tau_9^4.$$

\end{example}

\section{The Coxeter length function in $B_{n}$ by applying the generalized standard $OGS$}\label{cox-length} 

In this section  We give an explicit formula for $\ell(\pi)$ considering the generalized standard $OGS$ presentation, as it is described in Theorem \ref{ogs-bn}, and the presentation that is described in Theorem \ref{uvuv}.
We start with a lemma about the length of elements of the form $\tau_{k}^{-k}$. Then, we find the length of elements in the subgroup $\dot{S}_n$. Finally we give a formula for the length of a general element of $B_n$.

\begin{lemma}\label{tau-k-k} For $1\leq k\leq n$, let $\tau_{k}$ be the element of $B_n$ as it is defined in Definition \ref{tau}. Then, the following holds:
 $$\ell(\tau_{k}^{-k}) = \sum_{i=0}^{k-1} 2i-1 = k^{2}$$\\
\end{lemma}

\begin{proof}
By considering the normal form of $\tau_k^{-k}\in B_{n}$ as it is defined in Definition \ref{Normal form of $B_{n}$}:
$$\tau_{k}^{-k} = \prod_{i=0}^{k-1} \prod_{j=0}^{2i} s_{|i-j|} = s_{0} \cdot (s_{1}\cdot s_{0} \cdot s_{1}) \cdot (s_{2}\cdot s_{1} \cdot s_{0} \cdot s_{1} \cdot s_{2})  \cdots (s_{k-1}\cdots s_{1}\cdot s_{0} \cdot s_{1} \cdots s_{k-1}). $$
Hence
$$\ell(\tau_{k}^{-k}) = \sum_{i=0}^{k-1} 2i-1 = k^{2}$$
\end{proof}

\begin{theorem}\label{ell-sn}
Let $\pi$ be an element in the subgroup $\dot{S}_{n}$ of $B_{n}$, where by the generalized standard $OGS$ presentation of $\pi$ (as it is described in Theorem \ref{main-ogs-sn}):
$$\pi=\tau_{k_{1}}^{i_{k_{1}}} \cdot \tau_{k_{2}}^{i_{k_{2}}} \cdots \tau_{k_{m}}^{i_{k_{m}}}$$  such that
\begin{itemize}
    \item $ -k_{j} \leqslant i_{k_{j}} \leqslant k_{j}-1$, for every $1 \leqslant j \leqslant m$.
    \item $\sum_{j=1}^{m} i_{k_{j}}=0$.
    \item  $0\leqslant \sum_{j=r}^{m}i_{k_{j}}\leqslant  k_{r-1}$, for every $2\leq r\leq m$.
\end{itemize}

Then, $$\ell(\pi) = \sum_{j=1}^{m} k_{j} \cdot i_{k_{j}}.$$
 
\end{theorem}

\begin{proof}

Let  $\pi$ be an element of $\dot{S}_{n}$, presented by the generalized standard $OGS$ (as it is described in Theorem \ref{main-ogs-sn}) as follow:
$$\pi=\tau_{k_{1}}^{i_{k_{1}}} \cdot \tau_{k_{2}}^{i_{k_{2}}} \cdots \tau_{k_{m}}^{i_{k_{m}}}$$  such that
\begin{itemize}
    \item $ -k_{j} \leqslant i_{k_{j}} \leqslant k_{j}-1$, for every $1 \leqslant j \leqslant m$.
    \item $\sum_{j=1}^{m} i_{k_{j}}=0$.
    \item  $0\leqslant \sum_{j=r}^{m}i_{k_{j}}\leqslant  k_{r-1}$, for every $2\leq r\leq m$.
\end{itemize}
The presentation of $\pi$ by the normal form (\cite{BB} Chapter 3.4) as follow: 
$$\pi = (s_{r_{1}} \cdot s_{r_{1}-1} \cdots s_{r_{1}-\mu_{1}} )\cdot (s_{r_{2}} \cdot s_{r_{2}-1} \cdots s_{r_{2}-\mu_{2}}) \cdots (s_{r_{z}} \cdot s_{r_{z}-1} \cdots s_{r_{z}-\mu_{z}}),$$ 
where $r_j<r_{j+1}$ for every $1\leq j\leq z-1$.
Therefore, \quad $\ell(\pi) = \sum_{j=1}^{z} \mu_{j} + z $.\\
Then, By Lemma \ref{tau-i+i},  $$ \pi = \tau_{r_{1}}^{-(\mu_{1}+1)} \cdot \tau_{r_{1}+1}^{\mu_{1}+1} \cdot \tau_{r_{2}}^{-(\mu_{2}+1)} \cdot \tau_{r_{2}+1}^{\mu_{2}+1} \cdots \tau_{r_{z}}^{-(\mu_{z}+1)}\cdot \tau_{r_{z}+1}^{\mu_{z}+1}.$$ 

Thus, $$\ell(\pi) = \sum_{j=1}^{z} (\mu_{j}+1) = \sum_{j=1}^{z} {-(\mu_{j}+1)\cdot r_{j} + (\mu_{j}+1) \cdot (r_{j}+1)}.$$

Since $r_j<r_{j+1}$, we conclude $r_{j}+1\leq r_{j+1}$.
Hence, we have either $r_j+1=r_{j+1}$  or $r_j+1<r_{j+1}$ for some $1\leq j\leq z-1$.
Therefore, in case $r_j+1=r_{j+1}$ for some $1\leq j\leq z-1$,
$$ \pi = \tau_{r_{1}}^{-(\mu_{1}+1)} \cdot \tau_{r_{1}+1}^{\mu_{1}+1} \cdot \tau_{r_{2}}^{-(\mu_{2}+1)} \cdot \tau_{r_{2}+1}^{\mu_{2}+1} \cdots \tau_{r_{z}}^{-(\mu_{z}+1)}\cdot \tau_{r_{z}+1}^{\mu_{z}+1}=$$
$$=\tau_{r_{1}}^{-(\mu_{1}+1)} \cdots \tau_{r_{j}+1}^{(\mu_{j}+1)-(\mu_{j+1}+1)} \cdots \tau_{r_{z}+1}^{\mu_{z}+1}=$$
$$=\tau_{k_1}^{i_{k_1}}\cdots \tau_{k_m}^{i_{k_m}},$$

Then, by the uniqueness of presentation by the generalized standard $OGS$ (Theorem \ref{ogs-bn}), we conclude:

$$\ell(\pi)= \sum_{j=1}^{z} (\mu_{j}+1)\cdot r_{j} + (\mu_{j}+1) \cdot (r_{j}+1)=\sum_{j=1}^m k_j\cdot i_{k_j}.$$

\end{proof}

 The next theorem demonstrates the connection of the generalized standard $OGS$ presentation of an  element $\pi$ of $B_{n}$ (as it is described in Theorem \ref{ogs-bn}) to the Coxeter length of  $\pi$.

\begin{theorem}
 Let $\pi$ be an element of $B_{n}$, Consider the presentation of $\pi$ in the form $u_{1} \cdot  v_{1} \cdot u_{2} \cdot v_{2} \cdots u_{r}$ for some $r$, as it is described in Theorem \ref{uvuv}.\\
Then, $$\ell(\pi) = \ell(u_{r}) + \ell(v_{r-1}) - \ell(u_{r-1}) - \ell(v_{r-1}) + \ell(u_{r-2}) + \ell(v_{r-2}) - \cdots + (-1)^{r-1} \cdot \ell(u_{1}).$$
\end{theorem}

\begin{proof}
First, consider the case:
$$\pi = u\cdot v=  \prod_{j=1}^{m} \tau_{k_{j}}^{i_{k_j}}\cdot \tau_{p}^{p}.$$
where $p\geq k_m$.
For $1\leq j<n$, denote by $\dot{B}_j$ the parabolic subgroup of $B_n$ which is generated by $s_0, s_1, \ldots s_{j-1}$. (Notice $\dot{B}_j$ is isomorphic to $B_j$.)
By Theorem \ref{ogs-bn}, $u$ is in the parabolic subgroup $\dot{B}_{k_m}$ of $B_n$, and by \cite{BB} Chapter 3.4, the normal form of $v=\tau_p^p$ as follow
$$v=\tau_p^p =  s_{0}\cdot (s_{1}\cdot s_{0}\cdot s_{1}) \cdots (s_{p-1}\cdot s_{p-2} \cdots s_{0} \cdot s_1\cdots s_{p-1})  $$

Since, $p\geq k_m$, both elements $u$ and $v$ can be considered as elements of $\dot{B}_p$.

By \cite{BB} Chapter 3.4  $v=\tau_p^p$ is the longest element in the parabolic subgroup $\dot{B}_{p}$ of $B_n$. Therefore, by \cite{BB} Chapter 2.3,
$$\ell(\pi) = \ell(v) - \ell(u).$$

Now, assume the following by induction on $m$:

If
$$ \pi = u_{1} \cdot v_{1} \cdots u_{m-1} \cdot v_{m-1} \cdot u_{m} \cdot v_{m},$$
then
$$\ell(\pi) = \sum_{i=1}^{m} (-1)^{m-i} \cdot (\ell(v_{i}) - \ell(u_{i})).$$

Now, consider :
$$\pi_{1} =  \pi \cdot u_{m+1}.$$
where $\pi=u_{1} \cdot v_{1} \cdots u_{m-1} \cdot v_{m-1} \cdot u_{m} \cdot v_{m}$. 
Then, $v_m=\tau_q^{q}$ for some $q<n$, and by Theorem \ref{uvuv} $\pi$ is in the parabolic subgroup $\dot{B}_q$ of $B_n$. Notice, that the generalized standard $OGS$ presentation of $u_{m+1}$ is  $u_{m+1}=\prod_{j=1}^m\tau_{k_j}^{i_{k_j}}$, where by Theorem \ref{uvuv}, $k_1\geq q$, and then by \cite{BB} Chapter 3.4, the normal form of $u_{m+1}$ as follow:
$$u_{m+1} = (s_{r_{1}} \cdot s_{r_{1}-1} \cdots s_{r_{1}-\mu_{1}} )\cdot (s_{r_{2}} \cdot s_{r_{2}-1} \cdots s_{r_{2}-\mu_{2}}) \cdots (s_{r_{z}} \cdot s_{r_{z}-1} \cdots s_{r_{z}-\mu_{z}})$$ 
where $r_j<r_{j+1}$ for every $1\leq j\leq z-1$.
By the proof of Theorem \ref{ell-sn}, $r_1=k_1$, and since $k_1\geq q$, we have that $r_1\geq q$.
Since, $\pi\in \dot{B}_q$, the normal form of $\pi\cdot u_{m+1}$ by \cite{BB} Chapter 3.4 as follow:

$$norm(\pi\cdot u_{m+1})=$$ $$=norm(\pi)\cdot (s_{r_{1}} \cdot s_{r_{1}-1} \cdots s_{r_{1}-\mu_{1}} )\cdot (s_{r_{2}} \cdot s_{r_{2}-1} \cdots s_{r_{2}-\mu_{2}}) \cdots (s_{r_{z}} \cdot s_{r_{z}-1} \cdots s_{r_{z}-\mu_{z}})=$$ $$=norm(\pi)\cdot norm(u_{m+1}).$$

Hence, $$\ell(\pi\cdot u_{m+1})=\ell(\pi)+\ell(u_{m+1}).$$

Now, consider:
$$\pi_{2} = \pi_{1}\cdot v_{m+1}=\pi \cdot (u_{m+1} \cdot v_{m+1}).$$

By the same argument as in the calculation of the length of $ (u\cdot v)$, we get that  $v_{m+1}$ is the longest element in a parabolic subgroup of $B_n$, which contains the elements $u_1, v_1, \ldots, u_m, v_m, u_{m+1}$. 
Hence, by \cite{BB} Chapter 2.3,

$$\ell(\pi_{2}) = \ell(v_{m+1})-\ell(\pi_1) =  \ell(v_{m+1})-[\ell(u_{m+1})+\ell(\pi)] =(\sum_{i=1}^{m+1} (-1)^{m+1-i}\cdot (\ell(v_{i}) - \ell(u_{i})).$$
 
 Hence, 
 
 $$\ell(\pi) = \sum_{i=1}^{m+1} (-1)^{m+1-i} (\ell(v_{i}) - \ell(u_{i})).$$
\end{proof}

\begin{example}
Let $\pi$ be: $$\pi = \tau_{2} \cdot \tau_{3} \cdot \tau_{4}^{3} \cdot \tau_{5}^{2}$$
$$\pi = \tau_{1} \cdot (\tau_{1}^{-1} \cdot \tau_{2} ) \cdot \tau_{2}^{2} \cdot (\tau_{2}^{-2} \cdot \tau_{3} \cdot \tau_{4}) \cdot \tau_{4}^{4} \cdot (\tau_{4}^{-2} \cdot \tau_{5}^{2}).$$
Hence,

 $$u_1=1\quad\quad v_1=\tau_{1}=\tau_{1}^{-1}\quad\quad u_2=\tau_{1}^{-1} \cdot \tau_{2}\quad \quad v_2=\tau_{2}^2=\tau_{2}^{-2}$$ 
  $$u_3=\tau_{2}^{-2} \cdot \tau_{3} \cdot \tau_{4}\quad\quad
  v_3=\tau_{4}^4=\tau_{4}^{-4}\quad\quad
  u_4=\tau_{4}^{-2} \cdot \tau_{5}^{2}.$$

Then,
$$\ell(\pi) = \ell(\tau_{4}^{-2} \cdot \tau_{5}^{2}) + \ell(\tau_{4}^{-4}) - \ell(\tau_{2}^{-2} \cdot \tau_{3} \cdot \tau_{4}) -\ell(\tau_{2}^{-2})+ \ell(\tau_{1}^{-1} \cdot \tau_{2})+\ell(\tau_{1}^{-1}).$$
$$= 4\cdot (-2) + 5\cdot 2 + 4\cdot 4 - 2\cdot (-2) - 3 - 4 - 2\cdot 2 + 1 \cdot (-1) + 2 + 1 = 13.$$

Indeed, the presentation of $\pi$ in the normal form contains $13$ Coxeter generators as follow

$$\text{norm}(\pi)=s_0\cdot s_1\cdot (s_2\cdot s_1\cdot s_0)\cdot (s_3\cdot s_2\cdot s_1\cdot s_0\cdot s_1\cdot s_2)\cdot (s_4\cdot s_3).$$

\end{example}

\section{The descent set of elements of $B_{n}$.}\label{descent-bn}
\quad

In this subsection, we characterize the descent set of the elements in the Coxeter group $B_{n}$ by using the generalized standard $OGS$ presentation for the elements of $B_n$

\begin{proposition}\label{des-bn-sn}
 Let $\pi$ be an element of $\dot{S}_{n}$ in $B_{n}$ which is expressed by the generalized standard OGS presentation (i.e., $\pi = \tau_{k_{1}}^{i_{k_{1}}} \cdot \tau_{k_{2}}^{i_{k_{2}}} \cdots \tau_{k_{m}}^{i_{k_{m}}}$, where by Theorem \ref{main-ogs-sn}, $\sum_{j=1}^{m} i_{k_j}=0$ , and $k_j\leqslant \sum_{j=1}^{r} i_{k_j} \leqslant 0$ $\quad$, $r < m.$)\\

Then $des(\pi)$, the set of the locations of the descents of $\pi$ satisfies the following property:

$$s_{k_{j}} \in des(\pi) \Longleftrightarrow  i_{k_{j}} < 0$$
\end{proposition}

\begin{proof}
Consider  $$\pi = \tau_{k_{1}}^{i_{k_{1}}} \cdot \tau_{k_{2}}^{i_{k_{2}}} \cdots \tau_{k_{m}}^{i_{k_{m}}} , \quad \text{where} \quad \sum_{j=1}^{m} i_{k_{j}}=0.$$
Denote by $r_1, \ldots, r_{q}$, all the indices $j$ such that $i_{k_{j}}<0$.  Since by Theorem \ref{main-ogs-sn},  $i_{k_1}<0$ for every element of $\dot{S}_n$, we have $r_1=1$.\\
 
 Now, for every $1\leq \alpha\leq q$,   let define $i^{\prime}_{\alpha}$ as it was defined in Theorem \ref{factorization-},
by the following recursive method:
\begin{itemize}
    \item Let define $i^{\prime}_{q}$ to be $$i^{\prime}_{q}=i_{k_m};$$
    \item starting with $\alpha=q-1$, let define $i^{\prime}_{q-1}$ to be
    $$i^{\prime}_{q-1}=i_{r_q}+\sum_{j=r_{q}+1}^{m-1} i_{k_{j}}+i^{\prime}_{q};$$
    \item Then for every $1\leq \alpha\leq q-2$,  let define $i^{\prime}_{\alpha}$ to be
    $$i^{\prime}_{\alpha}=i_{r_{\alpha+1}}+\sum_{j=r_{\alpha+1}+1}^{r_{\alpha+2}-1} i_{k_{j}}+i^{\prime}_{\alpha+1}.$$
\end{itemize}

Then by Theorem \ref{factorization-}, 

$$\pi = (\tau_{k_{r_1}}^{-\sum_{j=2}^{r_2-1} i_{k_{j}}-i^{\prime}_{1}} \cdot \tau_{k_{2}}^{i_{k_{2}}}\cdots  \tau_{k_{r_{2}-1}}^ {i_{k_{r_{2}-1}}}\cdot\tau_{k_{r_{2}}}^{i^{\prime}_{1}}) \cdot (\tau_{k_{r_{2}}}^{-\sum_{j=r_{2}+1}^{r_3-1} i_{k_{j}}-i^{\prime}_{2}} \cdot \tau_{k_{r_{2}+1}}^{i_{k_{r_{2}+1}}}  \cdots \tau_{k_{r_{3}-1}}^{i_{k_{r_{3}-1}}}\cdot \tau_{k_{r_{3}}}^{i^{\prime}_{2}})\cdots$$ $$\cdots (\tau_{k_{r_{q}}}^{-\sum_{j=r_{q}+1}^{m-1} i_{k_{j}}-i^{\prime}_{q}} \cdot \tau_{k_{r_{q}+1}}^{i_{k_{r_{q}+1}}}  \cdots \tau_{k_{m-1}}^{i_{k_{m-1}}} \tau_{k_{m}}^{i^{\prime}_{q}}).$$

Then, by Theorem \ref{factorization-} the following properties hold:
\begin{itemize}
    \item By considering $\pi$ as an element of $S_n$, the subword
    $$\pi_{\alpha}=\tau_{k_{r_{\alpha}}}^{-\sum_{j=r_{\alpha}+1}^{r_{\alpha+1}-1} i_{k_{j}}-i^{\prime}_{\alpha}} \cdot \tau_{k_{r_{\alpha}+1}}^{i_{k_{r_{\alpha}+1}}}  \cdots \tau_{k_{r_{\alpha+1}-1}}^{i_{k_{r_{\alpha+1}-1}}}\cdot \tau_{k_{r_{\alpha+1}}}^{i^{\prime}_{\alpha}}$$ is the standard $OGS$ elementary factor $\pi^{(\alpha)}$ of $\pi$ for every $\alpha$ such that  $1\leq \alpha\leq q$, where $maj(\pi^{(\alpha)})=k_{r_{\alpha}}$, and $z(\pi)=q$;
    \item $i_{k_j}>0$ for every $2\leq j\leq m$ such that $j\neq r_{\alpha}$ for $1\leq \alpha\leq q$.
\end{itemize}
Hence, by Theorem \ref{theorem-factorization}, 
$$des\left(\pi\right)=\bigcup_{\alpha=1}^{z(\pi)}des\left(\pi^{(\alpha)}\right)=\{maj\left(\pi^{(\alpha)}\right)~|~1\leq \alpha\leq z(\pi)\};$$
Hence,  for $1\leq j\leq k_m$, ~$s_j\in des(\pi)$ if and only if $j=k_{r_{\alpha}}$ for $1\leq \alpha\leq q$. Since by definition of $r_{\alpha}$,  ~$j=k_{r_{\alpha}}$ if and only if $i_j<0$, we get that $s_j\in des(\pi)$ if and only if $i_j<0$.
\end{proof}
\begin{example}
 Let $\pi=\tau_{3}^{-2} \cdot \tau_{4}^{-1} \cdot \tau_{5}^{3}$ .\\

Then, the following holds:
\begin{itemize}
    \item $k_1=3, \quad k_2=4, \quad k_3=5$;
    \item $i_{k_{1}} =-2, \quad i_{k_2}=-1, \quad i_{k_3}=3$;
    \item $i_{k_1}=-2<0,\quad i_{k_1}+i_{k_2}=-1\leq 0,\quad i_{k_1}+i_{k_2}+i_{k_3}=0$.
    \end{itemize}

Hence, by Theorem \ref{main-ogs-sn} $\pi\in \dot{S}_5$.
\\

The permutation presentation of $\pi$ is as follows:

$$\pi=\tau_{3}^{-2} \cdot \tau_{4}^{-1} \cdot \tau_{5}^{3}$$
$$=
[3; ~-1; ~-2; ~4; ~5] 
\cdot 
[2; ~3; ~4; ~-1; ~5]
\cdot
[-3; ~-4; ~-5; ~1; ~2]
$$
$$=
[1; ~4; ~5; ~3; ~2] 
$$

Then we have,
 $$\pi(1)<\pi(2)<\pi(3),\quad\quad \pi(3)>\pi(4)>\pi(5).$$

Hence, 
$$\text{des}(\pi)=\{3, 4\}.$$

\end{example}

\begin{proposition}
Let $\pi$ be the element $\tau_{k}^{-k}$ , Then: $$des(\tau_{k}^{-k}) = \{0,1,2,\cdots ,k-1\}$$
\end{proposition}
\begin{proof}
Consider the permutation presentation of $\pi$:
$$\pi =
[-1; ~-2; \ldots; ~-k; ~k+1; ~k+2; \ldots; ~n]
$$
Then,by definition $2.0.1.$
$$des(\tau_{k}^{-k}) = \{0,1,2,\cdots ,k-1\}$$
\end{proof}

\begin{theorem}
Let $\pi$ be the element $v\cdot u = \tau_{i_{1}}^{-i_{1}} \cdot (\tau_{i_{2}}^{k_{2}} \cdots \tau_{i_{m}}^{k_{m}}).$ where $v$ is the longest element of the parabolic subgroup which is generated by the Coxeter elements $\{s_{0},s_{1} ,\cdots , s_{i_{1}-1}\}$, and $u$ is an element of $\dot{S}_{n}$, presented by the generalized standard $OGS$, where $i_2\geq i_1$. Then the descent set of $\pi$ satisfies the following:
\begin{itemize}
    \item $j \in des(v\cdot u) \quad , 0 \leqslant j < i_{1} \Longleftrightarrow j \notin des (u).$
   \item $j \in des(v\cdot u) \quad , j > i_{1} \Longleftrightarrow j \in des(u).$
   \item $i_1\notin des(v\cdot u)$.
 \end{itemize}
\end{theorem}

\begin{proof}

 Consider
$$\pi = \tau_{i_{1}}^{-i_{1}} \cdot (\tau_{i_{2}}^{k_{2}} \cdots \tau_{i_{m}}^{k_{m}}).$$
$$ v = \tau_{i_{1}}^{-i_{1}}.$$
$$u = \tau_{i_{2}}^{k_{2}} \cdots \tau_{i_{m}}^{k_{m}} \quad , u \in S_{i_{m}}.$$
Then by considering the permutation presentation of the element $vu$:
\begin{itemize}
    \item  $[v\cdot u](j) = -u(j) \quad 1\leqslant j \leqslant i_{1}$
    \item $[v\cdot u](j) = u(j) \quad i_{1}+1\leqslant j \leqslant i_{m}.$
\end{itemize}
Hence the following holds:
\begin{itemize}
    \item For $1\leq j\leq i_{1}-1$: $$j\in des(v\cdot u)\Longleftrightarrow [v\cdot u](j)>[v\cdot u](j+1)\Longleftrightarrow u(j)<u(j+1) \quad  (i.e.,\quad j\notin des(u)).$$
    \item For $i_1+1\leq j\leq n-1$: $$j\in des(v\cdot u)\Longleftrightarrow [v\cdot u](j)>[v\cdot u](j+1)\Longleftrightarrow u(j)>u(j+1) \quad  (i.e., \quad j\in des(u)).$$
    \item For $j=i_1$: $$[v\cdot u](i_1)=-u(i_1), \quad [v\cdot u](i_1+1)=u(i_1+1) \Longrightarrow [v\cdot u](i_1)<[v\cdot u](i_1+1) \Longrightarrow i_1\notin des(v\cdot u).$$
\end{itemize}

\end{proof}

\begin{theorem}
 Let $\pi$ be the element $u\cdot v =(\tau_{i_{1}}^{k_{1}} \cdot \tau_{i_{2}}^{k_{2}} \cdots \tau_{i_{m}}^{k_{m}}) \cdot \tau_{i_{m+1}}^{i_{m+1}}$ where $v$ is the longest element of the parabolic subgroup generated by the Coxeter elements $\{s_{0},s_{1} ,\cdots , s_{i_{m+1}-1}\}$ ,and $u$ is element of $\dot{S}_{n}$ presented by the generalized standard $OGS$, and $i_m\leq i_{m+1}$, then the descent set of $\pi$ satisfies the following property:
 \begin{itemize}
     \item  $j \in des(u\cdot v) \quad , 0 \leqslant j <i_{m} \Longleftrightarrow j \notin des(u).$
    \item $j \in des(u\cdot v) \quad , i_{m}\leqslant j <i_{m+1}.$
 \end{itemize}
\end{theorem}

\begin{proof}

 Consider

$$\pi = (\tau_{i_{1}}^{k_{1}} \cdot \tau_{i_{2}}^{k_{2}} \cdots \tau_{i_{m}}^{k_{m}}) \cdot \tau_{i_{m+1}}^{i_{m+1}} \quad, i_{m}\leq  i_{m+1}.$$

 $$u=\tau_{i_{1}}^{k_{1}} \cdot \tau_{i_{2}}^{k_{2}} \cdots \tau_{i_{m}}^{k_{m}} \quad u \in S_{i_{m}}.$$
$$des(u) = i_{j} \quad , k_{j} < 0.$$

$$v=\tau_{i_{m+1}}^{i_{m+1}}.$$

$$des(v) = \{0,1,\cdots , i_{m+1-1}\}.$$

Then by looking at the permutation presentation of the element $uv$:
$$[u\cdot v](j) = -u(j)   \quad 1\leqslant j \leqslant i_m, \quad\quad [u\cdot v](j) = -j  \quad i_{m}< j < i_{m+1}.$$

Hence the following holds:
\begin{itemize}
    \item For $1\leq j\leq i_ m$: $$j\in des(u\cdot v)\Longleftrightarrow [u\cdot v](j)>[u\cdot v](j+1)\Longleftrightarrow u(j)<u(j+1) \quad  (i.e.,\quad j\notin des(u)).$$
    \item For $i_m < j\leq i_m+1$: $$[u\cdot v](j)>[u\cdot v](j+1)\Longrightarrow j\in des(u\cdot v)).$$
    
\end{itemize}

\end{proof}

\begin{example}
The case $\pi=v \cdot u$:\\

$\bullet$ Consider $$\pi=\tau_{9}^{4}\cdot \tau_{10}^{4}= \tau_{8}^{-8} \cdot (\tau_{8}^{-8} \cdot \tau_{9}^{4}\cdot \tau_{10}^{4}) .$$

The permutation presentation of $\pi$ as follow:

$$v= \tau_{8}^{-8}=
[-1; ~-2; ~-3; ~-4; ~-5; ~-6; ~-7; ~-8; ~9; ~10] 
$$
$$u=\tau_{8}^{-8} \cdot \tau_{9}^{4}\cdot \tau_{10}^{4}$$
$$=
[-1; ~-2; ~-3; ~-4; ~-5; ~-6; ~-7; ~-8; ~9; ~10]
\cdot
[-6; ~-7; ~-8; ~-9; ~1; ~2; ~3; ~4; ~5; ~10] 
\cdot $$
$$
[~-7; ~-8; ~-9; ~-10; ~1; ~2; ~3; ~4; ~5; ~6] 
$$
$$=
[2; ~3; ~4; ~5; ~7; ~8; ~9; ~10; ~1; ~6].
$$
Hence, $$\pi=v\cdot u=\tau_{8}^{-8} \cdot (\tau_{8}^{-8} \cdot \tau_{9}^{4}\cdot \tau_{10}^{4})=\tau_{9}^{4}\cdot \tau_{10}^{4}$$
$$=
[-1; ~-2; ~-3; ~-4; ~-5; ~-6; ~-7; ~-8; ~9; ~10]\cdot 
[2; ~3; ~4; ~5; ~7; ~8; ~9; ~10; ~1; ~6]
$$
$$=
[-2; ~-3; ~-4; ~-5; ~-7; ~-8; ~-9; ~-10; ~1; ~6]
$$
Then :\\
$$des(v)=des(\tau_{8}^{-8}) = \{0,1,2,...,7\}, \quad des(u)=des(\tau_{8}^{-8} \cdot \tau_{9}^{4}\cdot \tau_{10}^{4} ) = \{8\}.$$

Hence, we conclude that in the permutation presentation of $u$, we get:

$$u(1)<u(2)<...<u(8)\quad \quad u(8)>u(9)\quad\quad  u(9) < u(10).$$

And in the permutation presentation of $v$ , $$v(i)= -i \quad \text{for} \quad 1<i<8.$$

Hence, in the permutation presentation of $v\cdot u$ :
\begin{itemize}
    \item  For every $1 < i < 8  \Longleftrightarrow [v\cdot u](i) = -u(i).$
    \item  For $i=9,10 \Longleftrightarrow [v\cdot u](i)=u(i).$
    \item  Therefore  $u(8)<0 \quad and \quad u(9)>0$, which applies $[v\cdot u](8) < [v\cdot u](9).$
    \item Furthermore the following holds:
     $$0 > [v\cdot u](1) > [v\cdot u](2) > [v\cdot u](3) > \cdots > [v\cdot u](8),\quad [v\cdot u](9) < [v\cdot u](10).$$
\end{itemize}

Hence, the descent set of $\pi$ as follow:
$$des(\pi)=\{0,1,2,3,4,5,6,7\}. $$

\end{example}

\begin{example}

The case $\pi=u\cdot v$ :\\

$\bullet$Let $\pi$ be $$(\tau_{3}^{-1}\cdot \tau_{4}^{-2} \cdot \tau_{5}^{3})\cdot \tau_5^{-5}.$$
 $u= \tau_{3}^{-1}\cdot \tau_{4}^{-2} \cdot \tau_{5}^{3}\in \dot{S}_{5}$ and $v=\tau_5^{-5}$.
\\

The permutation presentation of $\pi$ as follow:
 
 $$u=\tau_{3}^{-1}\cdot \tau_{4}^{-2} \cdot \tau_{5}^{3}$$
 $$=
[2; ~3; ~-1; ~4; ~5]
\cdot
[3; ~4; ~-1; ~-2; ~5]
\cdot 
[-3; ~-4; ~-5; ~1; ~2]$$
$$=
[1; ~3; ~5; ~4; ~2].
$$

$$v=\tau_5^{-5}=
[-1; ~-2; ~-3; ~-4; ~-5].
$$

Hence, 
$$\pi=u\cdot v=(\tau_{3}^{-1}\cdot \tau_{4}^{-2} \cdot \tau_{5}^{3})\cdot \tau_5^{-5}$$
$$
[-1; ~-3; ~-5; ~-4; ~-2].
$$

 Then by Proposition \ref{des-bn-sn}:
  $$des(u)=\{3,4\}$$
Hence we conclude:
    $$u(1) < u(2) < u(3)\quad\quad   u(3) > u(4) > u(5).$$

Now, we turn to the element $v$. We know that:
$$v(i) = -i \quad \text{for every} \quad  1<i<5 .$$
 Then, Then by Proposition \ref{des-bn-sn}: $$ des(v) = \{0,1,2,3,4\}.$$

Therefore, in the permutation presentation of  $u\cdot v$, We have:\\
$$[uv](i) = -u(i) \quad 1<i<5.$$

Hence, we have in the permutation presentation of $u\cdot v$,

$$(0 > [u\cdot v](1) >[u\cdot v](2)  > [u\cdot v](3)), \quad ([u\cdot v](3) < [u\cdot v](4) < [u\cdot v](5) ).$$
 Then the descent set of $\pi=u\cdot v$  is as follow:

$$des(u\cdot v)= \{0,1,2\}\quad \text{where}\quad des(u\cdot v) = \{0,1,2,3,...,k-1\} - des(u).$$
\end{example}

\section{Conclusions and future plans}

This paper is a natural continuation of the paper \cite{S1}, where  there was introduced a quite interesting generalization of the fundamental theorem for abelian groups to two important and very elementary families of non-abelian Coxeter groups, the $I$-type (dihedral groups), and the $A$-type (symmetric groups).
There were introduced canonical forms, with very interesting exchange laws, and quite interesting properties concerning the Coxeter lengths of the elements. In this paper we generalized the results of \cite{S1} for the $B$-type Coxeter groups, namely $B_n$. The results of the paper  motivate us for further generalizations of the $OGS$ and the arising properties from it for more families of  Coxeter and generalized Coxeter groups, which have an importance in the classification
of Lie algebras and the Lie-type simple groups, and in other fields of mathematics, such as
algebraic geometry for classification of fundamental groups of Galois covers of surfaces \cite{alst}.
The next natural step of generalization is defining a generalized standard OGS canonical form for $D$-type Coxeter group, which similarly to the $B$-type Coxeter groups, have presentations
as signed permutations. Furthermore, it is interesting to find generalization of the $OGS$ to the affine classical families $\tilde{A}_n$, $\tilde{B}_n$, $\tilde{C}_n$, and $\tilde{D}_n$, and also to other generalizations of the mentioned Coxeter groups, as the complex reflection groups $G(r, p, n)$ \cite{ST} or the generalized affine classical groups, the definition of which is described in \cite{rtv}, \cite{ast}.

\end{document}